\documentclass[a4paper,11pt]{article}%

\pdfoutput=1 
\usepackage{amsmath}%
\usepackage{amsfonts}%
\usepackage{amssymb}%
\usepackage{amsthm}
\usepackage{stmaryrd}
\usepackage{graphicx}
\usepackage[utf8]{inputenc}
\usepackage[T1]{fontenc}
\usepackage[french, english]{babel}
\usepackage{enumitem}
\usepackage{tikz}
\usetikzlibrary{decorations.markings}
\usepackage[colorlinks=true]{hyperref}

\newtheorem{thm}{Theorem}

\newtheorem{theorem}{Theorem}

\newtheorem{claim}[theorem]{Claim}

\newtheorem{lemma}[theorem]{Lemma}

\newtheorem{proposition}[theorem]{Proposition}

\theoremstyle{definition}
\newtheorem{definition}[theorem]{Definition}

\theoremstyle{remark}
\newtheorem{example}[theorem]{Example}
\newtheorem{remark}[theorem]{Remark}

\newlength{\espaceavantspecialthm}
\newlength{\espaceapresspecialthm}
\newcommand{\R}{\mathbb{R}}
\newcommand{\N}{\mathbb{N}}

\newcommand{\Z}{\mathbb{Z}}
\newcommand{\T}{\mathbb{T}}
\newcommand{\Sp}{\mathbb{S}}

\newcommand{\M}{\mathcal{M}}

\newcommand{\varep}{\varepsilon}

\newcommand{\conv}{\operatorname{conv}}
\newcommand{\Homeo}{\operatorname{Homeo}}
\newcommand{\card}{\operatorname{Card}}

\newcommand{\Id}{\operatorname{Id}}

\newcommand{\Span}{\operatorname{span}}
\newcommand{\ud}{\,\mathrm{d}}
\newcommand{\inte}{\operatorname{int}}

{\\}

\newenvironment{defi*}[1][]{
\vskip \espaceavantspecialthm \noindent \textbf{D\'efinition.} }%
{\vskip \espaceapresspecialthm}

\counterwithin{equation}{section}

\tikzset{->-/.style={decoration={
  markings,
  mark=at position .5 with {\arrow{latex}}},postaction={decorate}}}

\newcommand\test[1]{
\pgfmathsetmacro{\var}{#1}
\pgfmathparse{ifthenelse(\var>=0,"positif","négatif")} \pgfmathresult}%

\newcommand{\hgline}[3]{
\pgfmathsetmacro{\thetaone}{#1}
\pgfmathsetmacro{\thetatwo}{#2}
\pgfmathsetmacro{\theta}{(\thetaone+\thetatwo)/2}
\pgfmathsetmacro{\phi}{abs(\thetaone-\thetatwo)/2}
\pgfmathsetmacro{\close}{less(abs(\phi-90),0.0001)}
\ifdim \close pt = 1pt
    \draw[->-, color=#3] (\thetaone:1) -- (\thetatwo:1);
\else
	\pgfmathsetmacro{\R}{tan(\phi)}
	\pgfmathsetmacro{\test}{(\thetaone-\thetatwo)/abs(\thetaone-\thetatwo)}
	\draw[->-, color=#3] (\thetaone:1) arc (\thetaone+\test*90:\thetaone+\test*(270-2*\phi):\R);
\fi
}

\newcommand{\hglinefill}[3]{
\pgfmathsetmacro{\thetaone}{#1}
\pgfmathsetmacro{\thetatwo}{#2}
\pgfmathsetmacro{\theta}{(\thetaone+\thetatwo)/2}
\pgfmathsetmacro{\phi}{abs(\thetaone-\thetatwo)/2}
\pgfmathsetmacro{\close}{less(abs(\phi-90),0.0001)}
\ifdim \close pt = 1pt
    \filldraw[->-, color=#3] (\thetaone:1) -- (\thetatwo:1) ;
\else
	\pgfmathsetmacro{\R}{tan(\phi)}
	\pgfmathsetmacro{\test}{(\thetaone-\thetatwo)/abs(\thetaone-\thetatwo)}
	\filldraw[color=#3, opacity=.2] (\thetaone:1) arc (\thetaone+\test*90:\thetaone+\test*(270-2*\phi):\R) arc (\thetaone+\test*(270-2*\phi)-90:\thetaone+\test*90+90:1);
\fi
}

\begin{document}

\sloppy

\title{Hyperbolic isometries of the fine curve graph of higher genus surfaces}
\author{Pierre-Antoine Guih\'eneuf, Emmanuel Militon}
\maketitle

\begin{abstract}
We prove that for a homeomorphism $f$ that is isotopic to the identity on a closed hyperbolic surface, the following are equivalent:
\begin{itemize}
\item $f$ acts hyperbolically on the fine curve graph;
\item $f$ is isotopic to a pseudo-Anosov map relative to a finite $f$-invariant set;
\item the ergodic homological rotation set of $f$ has nonempty interior.
\end{itemize}
\end{abstract}

\selectlanguage{english}

\setcounter{tocdepth}{1}
\tableofcontents

\section{Introduction}\label{Sec:intro}

Let $S$ be a closed orientable surface with genus $\geq 2$. We denote $\Homeo(S)$ the space of homeomorphisms of $S$, and $\Homeo_0(S)$ the connected component of $\Id_S$ in $\Homeo(S)$. 

The \emph{fine curve graph} $C^\dagger(S)$ of $S$ was introduced by Bowden, Hensel, and Webb \cite{BHW} to give a counterpart of the classical \emph{curve graph} adapted to the study of the group of all homeomorphisms of $S$.

\begin{definition}
The \emph{fine curve graph} on the surface $S$ is the graph $C^\dagger(S)$ whose vertices are essential\footnote{\emph{I.e.} non contractible.} simple loops. There is an edge between two vertices $\alpha$ and $\beta$ if and only if the loops $\alpha$ and $\beta$ are disjoint.
\end{definition}

As a consequence of the Gromov hyperbolicity of the classical curve graphs for punctured surfaces, it was proved in \cite{BHW} that the fine curve graph $C^\dagger(S)$ is Gromov hyperbolic. The authors also prove that this graph is connected and has infinite diameter. 
This enables them to use large scale geometry techniques to study $\Homeo(S)$ via its action on $C^\dagger(S)$. As an application, they prove that, for any closed surface $S$ of genus $\ge 1$, the commutator length and the fragmentation norm on $\mathrm{Homeo}_0(S)$ are unbounded, answering a question posed by Burago, Ivanov, and Polterovich \cite{zbMATH05526532}.
\medskip

In the same way as the mapping class group $\mathrm{Map}(S)$ acts on $C(S)$ by isometries, the whole homeomorphism group $\mathrm{Homeo}(S)$ acts on $C^\dagger(S)$ by isometries. Gromov has classified isometries of Gromov hyperbolic spaces (see \cite[paragraph 8]{zbMATH04031953} or \cite{zbMATH01385418}) according to the asymptotic translation length, defined for an isometry $g$ of a Gromov hyperbolic space $X$ as
\[|g|_X = \lim_{n\to+\infty}\frac{1}{n} d_X\big(x,g^n(x)\big).\]
It is a standard exercise to see that this limit exists and is independent of $x$. This independence immediately implies that the asymptotic translation length is a conjugacy invariant of isometries of $X$.
Gromov classification is then as follows: for $g$ an isometry of a Gromov hyperbolic space, $g$ is
\begin{itemize}
\item \textbf{Hyperbolic} (or loxodromic) if the asymptotic translation length is positive;
\item \textbf{Parabolic} if the asymptotic translation length is zero but $g$ has no
finite diameter orbits, and
\item \textbf{Elliptic} if $g$ has finite diameter orbits.
\end{itemize}
There is an equivalent reformulation of this trichotomy in terms of fixed points on the Gromov boundary of $X$, but we do not require this point of view in the present work.
\bigskip

The present work aims at starting the classification of the type of isometry actions of $f\in\Homeo_0(S)$ on $C^\dagger(S)$ in terms of rotational properties of $f$.\footnote{Note that by a theorem of Long, Margalit, Pham, Verberne and Yao \cite{long2021automorphisms}, any isometry of $C^\dagger(S)$ is induced by the action of some homeomorphism.}
Such a classification is now completed in the case of the torus due to works of Bowden,  Hensel, Mann, Militon and Webb \cite{BHMMW} and Guihéneuf and Militon \cite{guiheneuf2023parabolic}. In particular, we have the following.

\begin{theorem}[\cite{BHMMW}, Theorem 1.3]
Let $f\in\Homeo_0(\T^2)$.
The following are equivalent:
\begin{itemize}
\item $f$ acts hyperbolically on $C^\dagger(\T^2)$;
\item the \emph{rotation set} of $f$ has non-empty interior;
\item there is a finite, $f$-invariant set $F\subset\T^2$ such that the restriction of $f$ to $\T^2\setminus F$ represents a pseudo-Anosov mapping class.
\end{itemize}
\end{theorem}

In the present work we get a counterpart of this statement for closed surfaces of genus $g\ge 2$, in terms of ergodic rotation set $\rho_{erg}(f)\subset H_1(S,\mathbb{R})$ of $f$, which is defined in the next section.

\begin{thm}\label{maintheorem}
Let $f \in \mathrm{Homeo}_0(S)$.
The following are equivalent:
\begin{enumerate}
\item The following does not hold. Either $\rho_{erg}(f)$ is contained in a hyperplane of $H_1(S,\mathbb{R})$, or there exists two nonempty rational subspaces $E$ and $F$ of $H_1(S,\mathbb{R})$, which are orthogonal for the intersection form $\wedge$ (defined after the theorem), and such that $\rho_{erg}(f) \subset E \cup F$. 
\item The ergodic rotation set $\rho_{erg}(f)$ has nonempty interior in $H_1(S,\mathbb{R})$.
\item There exists a finite $f$-invariant subset $F \subset S$ such that $f$ is isotopic to a pseudo-Anosov homeomorphism relative to $F$ (this notion is recalled in Section~\ref{Sec:1implies2}).
\item The homeomorphism $f$ acts hyperbolically on $C^{\dagger}(S)$.
\end{enumerate}
\end{thm}

The intersection form $\wedge$ on $H_1(S,\mathbb{R})$ can be obtained from the cup product on $H^1(S,\R) \times H^1(S,\R) \to  H^2(S,\R)$ that induces by Poincaré duality a symplectic form 
\[\wedge : H_1(S,\R) \times H_1(S,\R) \to \R.\]
This intersection form has a geometrical interpretation in the case of elements of $H_1(S,\Z)$: if $\alpha$ and $\beta$ are closed curves, then $[\alpha]_{H_1}\wedge [\beta]_{H_1}$ is the algebraic intersection number between $\alpha$ and $\beta$. For more details about these facts, see \cite[Section~1.1]{lellouch}.

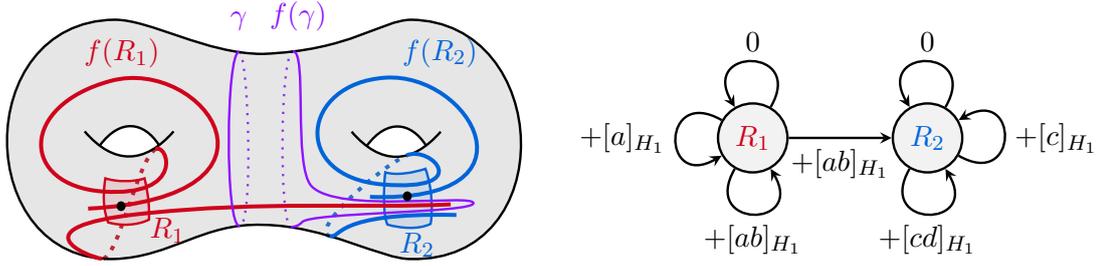
\begin{figure}[bt]
\tikzset{every picture/.style={line width=0.9pt}} 

\begin{tikzpicture}[x=0.75pt,y=0.75pt,yscale=-1,xscale=.95]

\draw [color={rgb, 255:red, 144; green, 19; blue, 254 }  ,draw opacity=1 ] [dash pattern={on 0.84pt off 2.51pt}]  (345.12,102.15) .. controls (338.23,103.09) and (337.94,189.51) .. (343.95,190.45) ;
\draw [color={rgb, 255:red, 144; green, 19; blue, 254 }  ,draw opacity=1 ] [dash pattern={on 0.84pt off 2.51pt}]  (316.97,102.46) .. controls (321.81,103.24) and (324.6,189.04) .. (317.27,189.82) ;
\draw [color={rgb, 255:red, 0; green, 100; blue, 217 }  ,draw opacity=1 ][line width=1.5]  [dash pattern={on 1.69pt off 2.76pt}]  (364.55,195) .. controls (353.83,192.49) and (388.11,162.13) .. (404.87,154.44) ;
\draw [color={rgb, 255:red, 208; green, 2; blue, 27 }  ,draw opacity=1 ][line width=1.5]  [dash pattern={on 1.69pt off 2.76pt}]  (272.91,151.34) .. controls (263.16,157.41) and (253.38,206.81) .. (244.76,206.45) ;
\draw [color={rgb, 255:red, 0; green, 0; blue, 0 }  ,draw opacity=1 ][fill={rgb, 255:red, 155; green, 155; blue, 155 }  ,fill opacity=0.25 ]   (329.26,189.39) .. controls (297.77,188.36) and (281.56,205.55) .. (251.31,206.58) .. controls (176.8,208.3) and (175.55,87.25) .. (251.31,86.22) .. controls (270.05,86.03) and (299.6,103.47) .. (329.26,103.42) .. controls (375.71,101.7) and (391.3,86.57) .. (407.2,86.22) .. controls (481.79,87.21) and (486.7,208.3) .. (407.2,206.58) .. controls (391.92,206.93) and (362.62,189.73) .. (329.26,189.39) -- cycle ;
\draw [draw opacity=0][fill={rgb, 255:red, 255; green, 255; blue, 255 }  ,fill opacity=1 ]   (243.8,149.34) .. controls (255.86,134.44) and (267.81,138.88) .. (274.84,149.34) .. controls (264.95,156.65) and (253.52,157.22) .. (243.8,149.34) -- cycle ;
\draw    (235.73,142.38) .. controls (251.09,159.45) and (267.35,158.96) .. (282.49,142.38) ;
\draw    (243.8,149.34) .. controls (254.49,135.7) and (266.82,137.42) .. (274.84,149.34) ;

\draw [draw opacity=0][fill={rgb, 255:red, 255; green, 255; blue, 255 }  ,fill opacity=1 ]   (384.1,149.51) .. controls (396.16,134.61) and (408.11,139.05) .. (415.14,149.51) .. controls (405.25,156.82) and (393.82,157.39) .. (384.1,149.51) -- cycle ;
\draw    (376.02,142.55) .. controls (391.39,159.62) and (407.65,159.13) .. (422.79,142.55) ;
\draw    (384.1,149.51) .. controls (394.79,135.87) and (407.12,137.59) .. (415.14,149.51) ;

\draw  [color={rgb, 255:red, 208; green, 2; blue, 27 }  ,draw opacity=1 ][fill={rgb, 255:red, 208; green, 2; blue, 27 }  ,fill opacity=0.1 ] (246.61,166.01) .. controls (252.14,167.8) and (261.35,167.98) .. (267.21,166.01) .. controls (269.83,176.75) and (269.59,181.26) .. (269.56,186.77) .. controls (263.19,189.81) and (254.65,189.99) .. (248.07,188.18) .. controls (246.44,184.26) and (244.73,177.25) .. (246.61,166.01) -- cycle ;
\draw  [color={rgb, 255:red, 0; green, 100; blue, 217 }  ,draw opacity=1 ][fill={rgb, 255:red, 0; green, 100; blue, 217 }  ,fill opacity=0.1 ] (393.7,162.97) .. controls (399.23,164.76) and (408.95,163.15) .. (414.81,161.18) .. controls (418.6,171.17) and (418.6,183.51) .. (417.49,188.92) .. controls (411.13,191.96) and (401.45,191.09) .. (394.88,189.27) .. controls (392.84,186.04) and (392.46,173.31) .. (393.7,162.97) -- cycle ;
\draw [color={rgb, 255:red, 208; green, 2; blue, 27 }  ,draw opacity=1 ][line width=1.5]    (237.56,181.22) .. controls (351.99,175.68) and (298.21,104.46) .. (249.46,116.99) .. controls (200.7,129.51) and (201.91,169.76) .. (245.44,172.63) .. controls (288.96,175.51) and (279.44,148.08) .. (272.91,151.34) ;
\draw [color={rgb, 255:red, 208; green, 2; blue, 27 }  ,draw opacity=1 ][line width=1.5]    (244.76,206.45) .. controls (237.26,205.53) and (199.53,190.71) .. (264.7,183.91) .. controls (329.87,177.11) and (383.32,181.22) .. (428.05,179.08) ;
\draw  [fill={rgb, 255:red, 0; green, 0; blue, 0 }  ,fill opacity=1 ] (253.14,180.01) .. controls (253.14,179) and (253.91,178.18) .. (254.86,178.18) .. controls (255.81,178.18) and (256.58,179) .. (256.58,180.01) .. controls (256.58,181.03) and (255.81,181.85) .. (254.86,181.85) .. controls (253.91,181.85) and (253.14,181.03) .. (253.14,180.01) -- cycle ;
\draw [color={rgb, 255:red, 0; green, 100; blue, 217 }  ,draw opacity=1 ][line width=1.5]    (385.66,175.32) .. controls (478.64,178.54) and (438.1,112.51) .. (393.7,115.38) .. controls (349.31,118.24) and (343.64,165.25) .. (392.28,169.07) .. controls (440.91,172.89) and (417.94,149.25) .. (404.87,154.44) ;
\draw [color={rgb, 255:red, 0; green, 100; blue, 217 }  ,draw opacity=1 ][line width=1.5]    (364.55,195) .. controls (370.75,196.61) and (365.73,183.01) .. (431.06,184.44) ;
\draw [color={rgb, 255:red, 144; green, 19; blue, 254 }  ,draw opacity=1 ]   (316.97,102.46) .. controls (309.94,101.52) and (309.2,191.39) .. (317.27,189.82) ;
\draw [color={rgb, 255:red, 144; green, 19; blue, 254 }  ,draw opacity=1 ]   (345.12,102.15) .. controls (353.03,101.37) and +(-10,-10) .. (355,170) .. controls +(14,14) and +(0,-5) .. (440,178.5) .. controls +(0,5) and +(5,-5) .. (350, 188) .. controls +(-2,2) and +(2,0) .. (343.95,190.45) ;
\draw  [fill={rgb, 255:red, 0; green, 0; blue, 0 }  ,fill opacity=1 ] (403.54,175.01) .. controls (403.54,174) and (404.31,173.18) .. (405.26,173.18) .. controls (406.21,173.18) and (406.98,174) .. (406.98,175.01) .. controls (406.98,176.03) and (406.21,176.85) .. (405.26,176.85) .. controls (404.31,176.85) and (403.54,176.03) .. (403.54,175.01) -- cycle ;

\draw (316.73,92.71) node [anchor=south] [inner sep=0.75pt]  [color={rgb, 255:red, 144; green, 19; blue, 254 }  ,opacity=1 ]  {$\gamma $};
\draw (348.04,93.47) node [anchor=south] [inner sep=0.75pt]  [color={rgb, 255:red, 144; green, 19; blue, 254 }  ,opacity=1 ]  {$f( \gamma )$};
\draw (268.4,184.4) node [anchor=north west][inner sep=0.75pt]  [color={rgb, 255:red, 208; green, 2; blue, 27 }  ,opacity=1 ]  {$R_{1}$};
\draw (255.32,111.98) node [anchor=south] [inner sep=0.75pt]  [color={rgb, 255:red, 208; green, 2; blue, 27 }  ,opacity=1 ]  {$f( R_{1})$};
\draw (410.39,191.6) node [anchor=north] [inner sep=0.75pt]  [color={rgb, 255:red, 0; green, 100; blue, 217 }  ,opacity=1 ]  {$R_{2}$};
\draw (401,111.58) node [anchor=south west] [inner sep=0.75pt]  [color={rgb, 255:red, 0; green, 100; blue, 217 }  ,opacity=1 ]  {$f( R_{2})$};

\end{tikzpicture}
\hfill
\begin{tikzpicture}
\node[draw,circle,fill=gray!10] (1) at (0,0) {\color{rgb, 255:red, 208; green, 2; blue, 27 }{$R_1$}};
\node[draw,circle,fill=gray!10] (2) at (2.3,0) {\color{rgb, 255:red, 0; green, 100; blue, 217 }{$R_2$}};

\draw[->, >=stealth] (1) to node[midway, below]{$+[ab]_{H_1}$} (2) ;
\draw[->, >=stealth] (1) to [looseness=5, out= 60, in=120]node[midway, above]{$0$} (1) ;
\draw[->, >=stealth] (1) to [looseness=5, out= 150, in=210]node[midway, left]{$+[a]_{H_1}$} (1) ;
\draw[->, >=stealth] (1) to [looseness=5, out= -120, in=-60]node[midway, below]{$+[ab]_{H_1}$} (1) ;
\draw[->, >=stealth] (2) to [looseness=5, out= 60, in=120]node[midway, above]{$0$} (2) ;
\draw[->, >=stealth] (2) to [looseness=5, out= -30, in=30]node[midway, right]{$+[c]_{H_1}$} (2) ;
\draw[->, >=stealth] (2) to [looseness=5, out= -120, in=-60]node[midway, below]{$+[cd]_{H_1}$} (2) ;

\end{tikzpicture}
\caption{\label{FigRotInterior}Example of a homeomorphism of the surface of genus 2 having homological rotation set with nonempty interior but homological ergodic rotation set included in the union of two planes. The black dots are contractible fixed points, the thick lines represent the images by $f$ of the rectangles $R_1$ and $R_2$; the intersections of all rectangles are supposed to be markovian.
}
\end{figure}

\begin{example}
No similar characterization hold with the more classical homological rotation set (in the sense of Definition~\ref{Def:homologicalrotationset}), as seen in the example of Figure~\ref{FigRotInterior} which is a homeomohphism of the suface of genus 2 having homological rotation set with nonempty interior but homological ergodic rotation set included in the union of two planes. \\
Indeed, the fact that a separating curve $\gamma$ is mapped to a curve disjoint from $f(\gamma)$ implies that the rotation vectors of recurrent points are included in the union of two planes that are the homologies of the surfaces respectively on the left and on the right of $\gamma$.\\
The rotation set of $f$ in the sense of Definition~\ref{Def:homologicalrotationset} contains the rotation set of the labeled Markov shift represented on the right (associated to the different markovian intersections between $R_1$, $R_2$ and their images by $f$). Here, $a$ and $b$ are generators of the fundamental group of the surface on the left of $\gamma$, and $c$ and $d$ are generators of the fundamental group of the surface on the right of $\gamma$.\\
Hence, this homeomorphism has rotation set with nonempty interior but by Theorem~\ref{maintheorem} does not act hyperbolically on $C^{\dagger}(S)$.
\end{example}

\subsection*{Plan of the proof}

The diagram of the proof of implications of theorem~\ref{maintheorem} is the following:

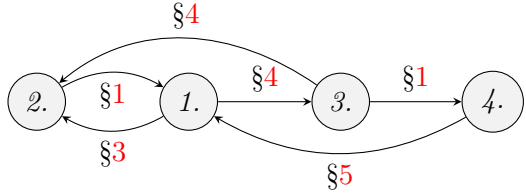
\begin{center}
\begin{tikzpicture}
\node[draw,circle,fill=gray!10] (1) at (0,0) {\textit{1.}};
\node[draw,circle,fill=gray!10] (2) at (-2,0) {\textit{2.}};
\node[draw,circle,fill=gray!10] (3) at (2,0) {\textit{3.}};
\node[draw,circle,fill=gray!10] (4) at (4,0) {\textit{4.}};

\draw[->, >=stealth] (1) to[bend left]node[midway, below]{\S\ref{Sec:Decomposition}} (2) ;
\draw[->, >=stealth] (2) to[bend left]node[midway, below]{\S\ref{Sec:intro}} (1) ;
\draw[->, >=stealth] (1) to node[midway, above]{\S\ref{Sec:1implies2}} (3) ;
\draw[->, >=stealth] (3) to node[midway, above]{\S\ref{Sec:intro}} (4) ;
\draw[->, >=stealth] (4) to[bend left]node[midway, below]{\S\ref{Sec:Last}} (1) ;
\draw[->, >=stealth] (3) to[out=145,in=40] node[midway, above]{\S\ref{Sec:1implies2}} (2) ;

\end{tikzpicture}
\end{center}

The implication \textit{2.} $\implies$ \textit{1.} is trivial.

The implication \textit{1.} $\implies$ \textit{2.} is a consequence of the shape of the ergodic rotation set that we study in Section~\ref{Sec:Decomposition} (Proposition~\ref{Prop:decomposition}). These results are based on a theorem by Lellouch \cite{lellouch} that we state as Theorem~\ref{Lellouch}. The fact that not only $\overline{\rho_{erg}(f)}$ but also $\rho_{erg}(f)$ has nonempty interior is done in Lemma~\ref{LemNonemInte} of Section~\ref{Sec:1implies2}.


The implication \textit{1.} $\implies$ \textit{3.} is a consequence of the Nielsen-Thurston classification. The proof is standard and similar to the ones due to Llibre-MacKay \cite{MR1101087} and Pollicott \cite{MR1094554}. We prove it in Section~\ref{Sec:1implies2}, using results of Section~\ref{Sec:Decomposition}.

The implication \textit{3.} $\implies$ \textit{4.} is standard and is a consequence of \cite[Lemma 4.2 ]{BHW} by Bowden, Hensel and Webb. Indeed, by this lemma, the asymptotic translation length of the action of $f$ on the fine curve graph of $S$ is at least the asymptotic translation length of the action of its mapping class on the curve graph of $S \setminus F$. But the latter does not vanish when it is a pseudo-Anosov element by \cite[Proposition 4.6]{zbMATH01355494} by Masur and Minsky.

The most difficult implication is \textit{4.} $\implies$ \textit{1.}. We will actually prove that if \textit{1.} does not hold, then $f$ cannot act hyperbolically on $C^{\dagger}(M)$. The last section of this article (Section~\ref{Sec:Last}) is devoted to the proof of this implication; it uses crucially the results of using results of Section~\ref{Sec:Decomposition}.

The next section (Section~\ref{Sec:homologicalrotation}) is devoted to some standard facts about the homological rotation set.

\subsection*{Acknowledgements}

The second author was supported by the ANR project Gromeov ANR-19-CE40-0007.

\section{Some useful results about the homological rotation set} \label{Sec:homologicalrotation}

Let us start by defining the homological rotation set that we use in this article. The definition we use here is different from the one which is used by Lellouch in \cite{lellouch}. More precisely, the latter rotation set is the convex hull of the rotation set that we define here, which is not necessarily convex. We also prove standard results about this homological rotation set which will be useful later.

We fix an isotopy $(f_t)_{t \in [0,1]}$ with $f_0=\Id_S$ and $f_1=f$. We extend this isotopy to $\mathbb{R}$ by setting, for any $t \in \mathbb{R}$, $f_t=f_{t-\lfloor t \rfloor}\circ f^{\lfloor t \rfloor}$, where $\lfloor t \rfloor$ is the lower integer part of $t$.

We endow the surface $S$ with a hyperbolic metric. For any points $x$ and $y$ of $S$, we fix a minimizing geodesic $g_{x,y}$ joining $x$ to $y$.

For any point $x$ in $S$ and any integer $n \geq 0$, we define a cycle by concatenating the path $(f_t(x))_{t \in [0,n]}$ with the path $g_{f^{n}(x),x}$. We denote by $c_{x,n} \in H_1(S,\mathbb{R})$ the homology class of this cycle.

\begin{definition} \label{Def:homologicalrotationset}
We call \emph{homological rotation set} of $f$ the subset of $H_1(S, \mathbb{R})$ which consists of limit values of sequences of the form $(\frac{1}{n}c_{x_n,n})$, where $(x_n)$ is a sequence of points of $S$.
\end{definition}

This homological rotation set is, by construction, a compact subset of $H_1(S, \mathbb{R})$.

Let $x \in S$. If the sequence $(\frac{1}{n}c_{x,n})_n$ converges, we say that the point $x$ admits a homological rotation vector. In this case, the limit $r(x)$ of the sequence $(\frac{1}{n}c_{x,n})_n$ is called the homological rotation vector of $x$. Observe that any periodic point of $f$ admits a homological rotation vector.

We can also associate rotation vectors to ergodic $f$-invariant measures, by the following lemma.

\begin{lemma} \label{Lem:ergodicrotationvector}
Let $\mu$ be an ergodic $f$-invariant Borel probability measure on $S$. Then there exists $r(\mu) \in \rho_{H_1}(f)$ such that $\mu$-almost any point $x$ of $S$ admits a rotation vector which is equal to $r(\mu)$. We call $r(\mu)$ the homological rotation vector associated to $\mu$.
\end{lemma}

We denote by $\rho_{erg}(f)$ the subset of $\rho_{H_1}(f)$ consisting of rotation vectors associated to ergodic measures.

\begin{proof} For any point $x \in S$, we denote by $I(x)$ the path $(f_t(x))_{t \in [0,1]}$. We see the cohomology group $H^1(S,\mathbb{R})$ as the quotient of the space of smooth closed $1$-forms on $S$ by the space of exact $1$-forms on $S$. The map $d$ from $H_1(S,\mathbb{R})$ to the algebraic dual  $\left(H^1(S,\mathbb{R}) \right)^*$, which, to an element $\gamma \in H_1(S,\mathbb{R})$, associates the linear form $\omega \mapsto \int_\gamma \omega$, is an isomorphism.

Hence it suffices to prove that for $\mu$-almost every point $x$ of $S$ and for any $\omega \in H^{1}(S,\mathbb{R})$,
$$ \lim_{n \rightarrow +\infty} \frac{1}{n} \int_{c(x,n)} \omega = \int_S \int_{I(y)}\omega \ud\mu(y).$$
Observe that the right-hand side of the above equality vanishes if $\omega$ is exact as the measure $\mu$ is $f$-invariant.

But this is a straightforward consequence of  Birkhoff ergodic theorem applied to the bounded  functions $x \mapsto \int_{I(x)}\omega$ and of the fact that 
$$\lim_{n \rightarrow +\infty} \frac{1}{n} \int_{g_{f^{n}(x),x}} \omega=0.$$
This last assertion holds as the length of $g_{f^{n}(x),x}$ is bounded from above by the diameter of $S$.
\end{proof}

\begin{lemma} \label{Lem:extremalpoints}
Any extremal point in $\mathrm{conv}(\rho_{H_1}(f))$ is a homological rotation vector associated to some ergodic $f$-invariant Borel probability measure on $S$.
\end{lemma}

\begin{proof}
We will use once again the isomorphism
$$\begin{array}{rrcl}
d: & H_1(S,\mathbb{R}) & \longrightarrow & \left( H^1(S,\mathbb{R}) \right)^* \\
 & \gamma & \longmapsto &  \left(\omega \mapsto \int_\gamma \omega\right).
 \end{array}$$
Moreover, we saw during the proof of Lemma \ref{Lem:ergodicrotationvector} that, for any ergodic invariant probability measure $\mu$,
\begin{equation}\label{EqDefr}
r(\mu)= d^{-1} \left(\omega \mapsto \int_S \int_{I(x)} \omega \ud\mu(x) \right).
\end{equation}
Let us denote by $\mathcal{M}(f)$ the set of $f$-invariant Borel probability measures on $S$.

Let 
$$ \rho_{mes}(f)=d^{-1}\left( \left\{ \omega \mapsto \int_S \int_{I(x)}\omega \ud\mu(x),\  \mu \in \mathcal{M}(f) \right\} \right).$$
This subset of $H_1(S,\mathbb{R})$ is the rotation set which is used in Lellouch's thesis \cite{lellouch}.

To prove the lemma, we need the following claim.

\begin{claim} \label{Cla:2definitions}
$$\mathrm{conv}(\rho_{H_1}(f))=\mathrm{conv}(\rho_{erg}(f))=\rho_{mes}(f).$$
\end{claim}

\begin{proof}[Proof of Claim \ref{Cla:2definitions}]
The right-hand side equality is a consequence of the fact that the set $\mathcal{M}(f)$ is the convex hull of the set of ergodic $f$-invariant probability measures on $S$, and the linearity of the map $d$ used to define $\rho_{mes}(f)$ (see \eqref{EqDefr}).

As, for any ergodic measure $\mu$, the vector $d^{-1} \left(\omega \mapsto \int_S \int_{I(x)}\omega \ud\mu(x) \right)$ belongs to $\rho_{erg}(f) \subset \rho_{H_1}(f)$ and as the set $\mathcal{M}(f)$ is the convex hull of ergodic measures, we obtain that
$$\rho_{mes}(f) \subset \mathrm{conv}(\rho_{H_1}(f)).$$

Conversely, fix $r \in \rho_{H_1}(f)$ and take a sequence of integers $n_k \rightarrow +\infty$ and a sequence $(x_k)$ of points of $S$ such that
$$ \lim_{k \rightarrow +\infty} \frac{1}{n_k} c_{x_k,n_k}=r.$$
Extracting a subsequence if necessary, we can suppose that the sequence $$(\mu_k)_k=\left( \frac{1}{n_k} \sum_{l=0}^{n_k-1} \delta_{f^{l}(x_k)} \right)_k$$ converges to a $f$-invariant probability measure $\mu$ for the weak-$*$ topology. Then, for any closed $1$-form $\omega$,
\begin{align*}
d(r)(\omega) & = \lim_{k \rightarrow +\infty} \frac{1}{n_k}\int_{c_{x_k,n_k}}\omega \\
 & = \lim_{k \rightarrow +\infty} \int_M \int_{I(x)} \omega \ud\mu_k \\
 & = \int_S \int_{I(x)} \omega \ud\mu(x).
\end{align*}
This proves that $r\in\rho_{mes}(f)$.
\end{proof}

We now finish the proof of Lemma \ref{Lem:extremalpoints}. 

Take any extremal point $r \in \mathrm{conv}(\rho_{H_1}(f))=\rho_{mes}(f)$. Observe that the subset
$$\left\{ \mu \in \mathcal{M}(f) \ | \ \omega \mapsto \int_M \int_{I(x)} \omega \ud\mu(x)=d(r) \right\}$$
of $\mathcal{M}(f)$ is convex. Hence any extremal point of this subset is an extremal point of $\mathcal{M}(f)$. Any such extremal point is an ergodic probability measure $\mu$ with $r(\mu)=r$.
\end{proof}

The following theorem is a straightforward consequence of Theorem C p.~16 of \cite{lellouch}. We denote by $\wedge$ the intersection form on $H_1(S, \mathbb{R})$.

\begin{theorem}[Lellouch] \label{Lellouch}
Let $r_1$ and $r_2$ be vectors in $\rho_{erg}(f)$. Suppose that $r_1 \wedge r_2 \neq 0$. Then, for any $r \in \mathrm{conv}\left\{0,r_1,r_2 \right\}$ and any $\epsilon >0$, there exists $r' \in \rho_{H_1}(f) \cap B(r,\epsilon)$  that is associated to a periodic point.  
\end{theorem}

\section{A decomposition of the rotation set} \label{Sec:Decomposition}

\begin{definition}
We define on $\rho_{erg}(f) \setminus \left\{ 0 \right\}$ the equivalence relation $\sim$ by:
$v \sim w$ if and only if there exists a sequence $(v_j)_{1 \leq j \leq m}\in\rho_{erg}(f)\setminus\{0\}$ such that $v_1=v$, $v_n=w$ and, for any $j <m$, either $v_j$ and $v_{j+1}$ are collinear or $v_j \wedge v_{j+1} \neq 0$.
\end{definition}

The fact that this is an equivalence relation is straightforward.

Observe that two distinct classes for this equivalence relation are orthogonal for the intersection form $\wedge$.

Taking the vector spaces generated by each of those equivalence classes, we obtain the following decomposition.

\begin{proposition} \label{Prop:decomposition}
Let $f\in \Homeo_0(S)$. 
There exist pairwise orthogonal vector subspaces $V_1,V_2, \ldots, V_n$ of $H_1(S,\mathbb{R})$, with $n \le g$, as well as a subset $\mathcal{L}$ of $H_1(S,\mathbb{R})$ such that:
\begin{enumerate}
\item $\displaystyle \rho_{erg}(f) \subset \mathcal{L} \cup \bigcup_{i=1}^n V_i.$
\item For any $1 \leq i \leq n$, $V_i$ is a rational subspace of $H_1(S,\mathbb{R})$ whose dimension is at least $2$.
\item The set $\mathcal{L}$ is a union of pairwise orthogonal lines containing $0$ which are all orthogonal to any $V_i$.
\item For any $i$, the set $\overline{V_i\cap \rho_{erg}(f)}$ is a convex set containing $0$, with nonempty interior in $V_i$, and with a dense subset of elements realized by periodic orbits.
\item For any $1 \leq i \leq n$ and any two vectors $v$, $w$ in $\rho_{erg}(f) \cap V_i \setminus \left\{ 0 \right\}$, $v \sim w$.
\end{enumerate} 
\end{proposition}

Figure~\ref{FigExuncount} displays an example where the ergodic rotation set is reduced to the set $\mathcal L$ but is uncountable and dense in a 2-dimensional convex subset of $H_1(S,\R)$. 

\begin{figure}
\begin{center}
\tikzset{every picture/.style={line width=1.2pt}} 

\begin{tikzpicture}[x=0.75pt,y=0.75pt,yscale=-1.6,xscale=1.8]

\draw [color={rgb, 255:red, 0; green, 0; blue, 0 }  ,draw opacity=1 ][fill={rgb, 255:red, 155; green, 155; blue, 155 }  ,fill opacity=0.25 ]   (300,150) .. controls (279.8,149.4) and (269.4,159.4) .. (250,160) .. controls (202.2,161) and (201.4,90.6) .. (250,90) .. controls (262.02,89.89) and (280.97,100.03) .. (300,100) .. controls (329.8,99) and (339.8,90.2) .. (350,90) .. controls (397.85,90.57) and (401,161) .. (350,160) .. controls (340.2,160.2) and (321.4,150.2) .. (300,150) -- cycle ;
\draw [draw opacity=0][fill={rgb, 255:red, 255; green, 255; blue, 255 }  ,fill opacity=1 ]   (245.18,126.71) .. controls (252.92,118.04) and (260.58,120.63) .. (265.09,126.71) .. controls (258.75,130.96) and (251.42,131.29) .. (245.18,126.71) -- cycle ;
\draw    (240,122.66) .. controls (249.86,132.59) and (260.29,132.3) .. (270,122.66) ;
\draw    (245.18,126.71) .. controls (252.04,118.78) and (259.95,119.78) .. (265.09,126.71) ;

\draw [draw opacity=0][fill={rgb, 255:red, 255; green, 255; blue, 255 }  ,fill opacity=1 ]   (335.18,126.81) .. controls (342.92,118.14) and (350.58,120.73) .. (355.09,126.81) .. controls (348.75,131.06) and (341.42,131.39) .. (335.18,126.81) -- cycle ;
\draw    (330,122.76) .. controls (339.86,132.69) and (350.29,132.4) .. (360,122.76) ;
\draw    (335.18,126.81) .. controls (342.04,118.88) and (349.95,119.88) .. (355.09,126.81) ;

\draw  [color={rgb, 255:red, 220; green, 155; blue, 170 }  ,draw opacity=0.5 ][line width=10]  (230.2,125.3) .. controls (230,148.9) and (269.4,154.5) .. (300,125) .. controls (330.6,95.5) and (370,100.4) .. (370,125) .. controls (370,149.6) and (329,154.1) .. (300,125) .. controls (271,95.9) and (230.4,101.7) .. (230.2,125.3) -- cycle ;
\draw  [color={rgb, 255:red, 208; green, 2; blue, 27 }  ,draw opacity=1 ] (230.2,125.3) .. controls (230,148.9) and (269.4,154.5) .. (300,125) .. controls (330.6,95.5) and (370,100.4) .. (370,125) .. controls (370,149.6) and (329,154.1) .. (300,125) .. controls (271,95.9) and (230.4,101.7) .. (230.2,125.3) -- cycle ;
\draw [color={rgb, 255:red, 208; green, 2; blue, 27 }  ,draw opacity=1 ]   (233.92,114.29) .. controls (231.04,118.35) and (231.34,118.67) .. (230.82,121.4) ;
\draw [shift={(230.13,124.31)}, rotate = 285.78] [fill={rgb, 255:red, 208; green, 2; blue, 27 }  ,fill opacity=1 ][line width=0.08]  [draw opacity=0] (8.04,-3.86) -- (0,0) -- (8.04,3.86) -- (5.34,0) -- cycle    ;
\draw [color={rgb, 255:red, 208; green, 2; blue, 27 }  ,draw opacity=1 ]   (367.67,115.54) .. controls (369.15,118.26) and (369.29,119.67) .. (369.58,122.04) ;
\draw [shift={(370,125)}, rotate = 260.19] [fill={rgb, 255:red, 208; green, 2; blue, 27 }  ,fill opacity=1 ][line width=0.08]  [draw opacity=0] (8.04,-3.86) -- (0,0) -- (8.04,3.86) -- (5.34,0) -- cycle    ;
\draw [color={rgb, 255:red, 74; green, 144; blue, 226 }  ,draw opacity=1 ]   (263.27,133.93) .. controls (281.06,127.92) and (284.23,115.27) .. (263.69,115.44) ;
\draw [shift={(261,115.53)}, rotate = 356.75] [fill={rgb, 255:red, 74; green, 144; blue, 226 }  ,fill opacity=1 ][line width=0.08]  [draw opacity=0] (8.04,-3.86) -- (0,0) -- (8.04,3.86) -- (5.34,0) -- cycle    ;
\draw [color={rgb, 255:red, 74; green, 144; blue, 226 }  ,draw opacity=1 ]   (337.67,133.4) .. controls (323.66,129.96) and (316.34,117.77) .. (336.15,115.52) ;
\draw [shift={(339.13,115.27)}, rotate = 176.75] [fill={rgb, 255:red, 74; green, 144; blue, 226 }  ,fill opacity=1 ][line width=0.08]  [draw opacity=0] (8.04,-3.86) -- (0,0) -- (8.04,3.86) -- (5.34,0) -- cycle    ;

\draw (279,122) node [anchor=west] [inner sep=0.75pt]  [color={rgb, 255:red, 74; green, 144; blue, 226 }  ,opacity=1 ]  {$a$};
\draw (322,122) node [anchor=east] [inner sep=0.75pt]  [color={rgb, 255:red, 74; green, 144; blue, 226 }  ,opacity=1 ]  {$c$};

\end{tikzpicture}

\caption{\label{FigExuncount}In this example, the homeomorphism $f$ is the time one of a time dependent vector field which is identically zero outside the light red neighbourhood of the red closed curve. The red closed curve represents the trajectory of a periodic orbit of $f$ under the isotopy $I$. One easily sees that the homological rotation set of $f$ must be included in $\langle [a]_{H_1}, [c]_{H_1}\rangle$, which is a totally isotropic subspace of $H_1(S,\R)$ for $\wedge$. However, by \cite[Theorem E]{guiheneuf2022homotopic}, the closure $\overline{\rho_{erg}(f)}$ of the ergodic rotation set has nonempty interior in this subspace. This gives an example where $\rho_{erg}(f) = \mathcal L$ but is uncontable.}
\end{center}
\end{figure}
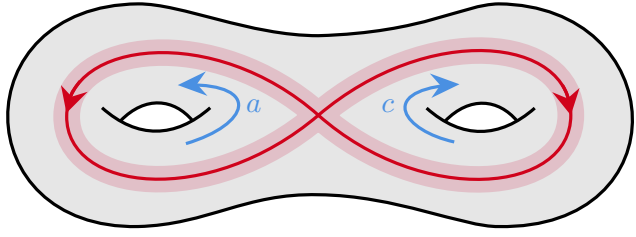

Note that a more precise decomposition of the ergodic homological rotation set, for transforming nonempty interior pieces in the set $\mathcal L$ into classes that are alike to the $V_i$ appearing in proposition~\ref{Prop:decomposition}, is the subject of a work in progress by Garc\'ia, Guihéneuf and Lessa \cite{GGL}.

\begin{proof}
Denote $(C_i)_{i\in I}$ the classes for the relation $\sim$ which generate a vector space with dimension at least $2$, and define $V_i=\operatorname{Span}(C_i)$. Define $\mathcal{L}$ as the union of the lines generated by the other classes.

Let us prove that that the family $(V_i)_{i \in I}$ is in direct sum. 
For $i\in I$, consider $v_i\in V_i$ and $\lambda_i\in \R$, and suppose that $\sum_{i\in J} \lambda_i v_i=0$, with $J$ a finite subset of $I$.
For some $j\in J$, by the definition of the $V_j$ there exists $w_j \in V_j$ such that $v_j \wedge w_j \neq 0$. As the spaces $V_i$ are pairwise orthogonal, we get that
\[0 = \left(\sum_{i\in J} \lambda_i v_i\right)\wedge w_j = \sum_{i\in J} \lambda_i (v_i\wedge w_j) = \lambda_j (v_j\wedge w_j),\]
and as $v_j\wedge w_j\neq 0$, this shows that $\lambda_j=0$. Hence, the family $(V_i)_{i \in I}$ is in direct sum, and the cardinality of $I$ can be at most $g$.
We denote this family of vector spaces by $(V_i)_{1 \leq i \leq n}$, with $n \leq g$.
\medskip

Point 1. of the proposition is immediate.
\medskip

Let us prove point 2. Take $1 \leq i \leq n$. Let $\mu_1,\dots,\mu_k\in \rho_{erg}(f)\setminus \left\{ 0 \right\}$ such that $(r(\mu_\ell))_{1\le\ell\le k}$ form a basis of the vector subspace $V_i$ associated to the equivalence class $C_i$. Then by definition of the equivalence relation $\sim$, for any $1\le \ell<k$, there exists $(\mu_\ell^j)_{1\le j\le m_\ell}\in \M_{erg}(f)$ such that $\mu_\ell^1 = \mu_\ell$, $\mu_\ell^{m_\ell} = \mu_{\ell+1}$ and $r(\mu_\ell^j)\wedge r(\mu_ \ell^{j+1})\neq 0$ for any $1\le j< m_\ell$. 

By Theorem \ref{Lellouch}, for any $\varep>0$, for any $1\le \ell\le k$, there exists a periodic measure $\nu_\ell$ such that $r(\nu_\ell)\in\rho_{H_1}(f)\cap B(r(\mu_\ell),\varep)$. By taking $\varep$ sufficiently small, one can moreover suppose that for $1\le \ell< k$ one has $r(\nu_\ell)\wedge r(\mu_\ell^2)\neq 0$, and $r(\nu_k)\wedge r(\mu_{k-1}^{m_{k-1}-1})\neq 0$. This implies that for any $1\le \ell\le k$, one has $r(\nu_\ell)\in C_i$. Moreover, by taking $\varep$ small enough, one can ensure that the family $(r(\nu_\ell))_{1\le\ell\le k}$ is free. Hence this family is free and has the same cardinality as a basis of $V_i$, so it is a basis of $V_i$, made of rational vectors.
\medskip

Point 3. follows from the construction of $\mathcal{L}$.
\medskip

 To prove point 4. it suffices to prove that, given $r(\mu_1),r(\mu_2)\in C_i$, then the whole triangle spanned by $r(\mu_1),r(\mu_2)$ and $0$ is accumulated by points of $C_i$ which are realized by periodic orbits. By definition, there exists $(\mu^j)_{1\le j\le m}\in \M_{erg}(f)$ such that $\mu^1 = \mu_1$, $\mu^{m} = \mu_2$ and $r(\mu^j)\wedge r(\mu^{j+1})\neq 0$ for any $1\le j< m$. Up to shortening the family if necessary, one can suppose that $r(\mu^j)\wedge r(\mu^{j+2})= 0$ for any $1\le j< m-1$.
Let $\varep>0$, and let us build a sequence of periodic measures $(\nu^j)_{2\le j\le m}$ by recurrence, satisfying:
\begin{itemize}
\item For any $2\le j \le m$, $r(\nu^j)\in B(r(\mu_1),\varep)$;
\item For any $2\le j \le m$, $r(\nu^j)\wedge r(\mu^j)\neq 0$.
\end{itemize}
Let us explain the first step of the construction, the other ones being identical.
By Theorem \ref{Lellouch}, as $r(\mu^2)\wedge r(\mu^1) \neq 0$, for any $0<\varep'<\varep/2$, the ball $B(\varep/2 r(\mu^2)+(1-\varep/2)r(\mu^1),\varep')$ contains a periodic measure $\nu^1$. If $\varep'$ is chosen small enough, then $r(\nu^1)\wedge r(\mu^3)$ is close enough to $(\varep/2 r(\mu^2)+(1-\varep/2)r(\mu^1))\wedge r(\mu^3) = \varep/2 r(\mu^2)\wedge r(\mu^3)\neq 0$, and in particular $r(\nu^1)\wedge r(\mu^3)\neq 0$. 

Finally, the measure $\nu^m$ satisfies $r(\nu^m)\in B(r(\mu_1),\varep)$ and $r(\nu^m)\wedge r(\mu_2)\neq 0$. It then suffices to apply Theorem \ref{Lellouch} to $\nu^m$ and $\mu_2$ to approximate any element of the triangle spanned by $r(\mu_1),r(\mu_2)$ and 0 by the rotation vector of a periodic measure, with distance proportional to $\varep$ (with the proportion coefficient depending on the norm of $r(\mu_1)$ and $r(\mu_2)$).
\medskip

Point 5. is a direct consequence of the definition of the vector subspaces $V_i$.
\end{proof}

Recall that the intersection form on $H_1(S, \mathbb{R})$ defines a symplectic form on $H_1(S, \mathbb{R})$. Recall also that a symplectic vector subspace of $H_1(S, \mathbb{R})$ is a subspace on which the restricted intersection form is a symplectic form. We will need the following lemma.

\begin{lemma}\label{LemNotContainedHyp}
Suppose that $\rho_{erg}(f)$ is not contained in a hyperplane. Then $\mathcal{L}= \emptyset$ and each of the subspaces $V_i$ which appear in the decomposition given by Proposition \ref{Prop:decomposition} is a symplectic subspace of $H_1(S,\mathbb{R})$.  
\end{lemma}

As a consequence, we deduce that each subspace $V_i$ is even-dimensional.

\begin{proof}
Observe that the set $\mathcal{L}$ has to be empty. Otherwise, the set $\rho_{erg}(f)$ would be contained in a hyperplane which consists of vectors which are orthogonal for the intersection form to one of the lines of $\mathcal{L}$. 

Fix a subspace $V_{i_0}$ given by Proposition \ref{Prop:decomposition}. We prove the following statement by induction on $n \geq 1$.
If $\mathrm{dim}(V_{i_0}) \geq 2n-1$, then there exists a family $(e_{i})_{1\le i \leq 2n}$ of independent vectors of $V_{i_0}$ such that
\begin{enumerate}
\item For any $1 \leq i \leq n$, $e_{2i-1} \wedge e_{2i}=1$ and $e_{2i} \wedge e_{2i+1}=0$.
\item For any $1 \leq i,j \leq 2n$, if $|i-j| \geq 2$, then $e_i \wedge e_j=0$.
\end{enumerate}
Suppose $n=1$. As the dimension of $V_{i_0}$ is at least $2$, we can find a nonzero vector $e_{1} \in C_{i_0}$. By definition of $C_{i_0}$ and as $V_{i_0}$ is at least $2$ dimensional by definition, there exists $e'_2 \in V_{i_0}$ such that $e_1 \wedge e'_2 \neq 0$. Changing $e'_2$ into an appropriate collinear vector, we can find $e_2$ such that $e_1 \wedge e_2=1$. This last relation implies that the vectors $e_1$ and $e_2$ are independent.

Suppose that $\mathrm{dim}(V_{i_0}) \geq 2n+1$ and that there exists a family $(e_{i})_{1\le i \leq 2n}$ of independent vectors of $V_{i_0}$ which satisfy the above conditions. By hypothesis on the dimension of $V_{i_0}$, we can find a vector $v$ of $V_{i_0}$ which does not belong to the subspace generated by $(e_i)_{i \leq 2n}$. Now let
$$ e_{2n+1}= v + \sum_{i=1}^{n} (v \wedge e_{2i-1}) e_{2i} - \sum_{i=1}^{n} (v \wedge e_{2i}) e_{2i-1}.$$
By construction of $e_{2n+1}$, we have $e_{2n+1} \neq 0$ and $e_{2n+1}$ is orthogonal to the space generated by the $e_i$'s, for $i\leq 2n$.
Now, we claim that there exists a vector $w \in V_{i_0}$ such that $e_{2n+1} \wedge w \neq 0$. Indeed, otherwise, $V_{i_{0}}$ would be contained in the orthogonal of $e_{2n+1}$ and hence, by Proposition \ref{Prop:decomposition}, $\rho_{erg}(f)$ would be contained in a hyperplane, a contradiction.

By the same trick used to construct $e_{2n+1}$, we can suppose that $w$ is orthogonal to all the vectors $e_{i}$, with $i \leq 2n$. Finally, we can find a vector $e_{2n+2}$ which is collinear to $w$ such that $e_{2n+1} \wedge e_{2n+2} =1$. The vectors $e_i$, for $1 \leq i \leq 2n+2$, have to be independent. 
\end{proof}

\section{Big rotation set homeomorphisms are pseudo-Anosov} \label{Sec:1implies2}

Suppose that case 1. of Theorem~\ref{maintheorem} holds.
By Proposition \ref{Prop:decomposition}, the set $\rho_{erg}(f)$ has nonempty interior in $H_1(S,\mathbb{R})$, and a dense subset of $\mathrm{Conv} \left(\rho_{erg}(f) \right)$ belongs to $\rho_{erg}(f)$ and is realized by a periodic orbit. We shall prove that in this case, the homeomorphism $f$ is isotopic to a pseudo-Anosov homeomorphism relative to a finite set. We will obtain it by classical arguments of \cite{MR1101087} and \cite{MR1094554}.

Let us first recall some definitions of Nielsen-Thurston theory.

Assume that $h\in\Homeo(\Sigma)$ is a homeomorphism of a compact surface $\Sigma$, possibly with boundary or punctures. We call $h$ \emph{periodic} if there exists $n > 0$ such that $h^n = \Id_\Sigma$. We call $h$ \emph{pseudo-Anosov} if there exist $h$-invariant measurable foliations with associated uniformly expanding transverse measures (see \cite{zbMATH04103989} for details). These two types of homeomorphism are distinct, and in particular periodic homeomorphisms have zero topological entropy while pseudo-Anosov homeomorphisms have nonzero topological entropy. Given a finite $h$-invariant set $F$ (in other words, $F$ is a finite union of $h$-periodic orbits), we call $h$ \emph{pseudo-Anosov relative to $F$} if $h|_{\Sigma\setminus F}$ is pseudo-Anosov.

The key theorem of Nielsen-Thurston theory is the following:

\begin{theorem}[Nielsen-Thurston classification]\label{TheoNTClass}
Every homeomorphism $f\in\Homeo(\Sigma)$ is isotopic to a homeomorphism $h\in\Homeo(\Sigma)$ such that:
\begin{enumerate}[label=(\roman*)]
\item $h$ leaves invariant a finite family (possibly empty) of disjoint simple closed curves $C_1, \dots, C_n$ on $\Sigma$;
\item No curve $C_i$ is homotopic to a boundary curve of $S$;
\item We can decompose $\Sigma = \bigcup_{j=1}^d \Sigma_j$, where $\Sigma_1,\dots, \Sigma_d$ are closed surfaces with disjoint interiors obtained by cutting the surface $S$ along the curves $\{C_1, \dots , C_n\}$;
\item For each $1 \le j \le  d$, the homeomorphism $h|_{\Sigma_j}$ is either periodic
or pseudo-Anosov.
\end{enumerate}
\end{theorem} 

In \cite{MR1101087}, Llibre and MacKay prove that if $S=\T^2$ and if $f\in \Homeo_0(S)$ has a rotation set with nonempty interior, then $f$ is isotopic to a pseudo-Anosov relative to a finite set. 

In \cite{MR1094554}, Pollicott shows the following:

\begin{theorem}[\cite{MR1094554}, Theorem 2]
Let $S$ be a compact closed surface of genus $g \ge 2$, and $f\in \Homeo_0(S)$. Assume that there exist $2g+ 1$ periodic points $x_1,\dots,x_{2g+1}$
whose rotation vectors $\rho_1,\dots,\rho_{2g+1}\in H_1(S,\R)$ do not lie on a hyperplane, then $f$ is isotopic to some $h\in \Homeo_0(S)$ having an invariant set $S_1\subset S$ that is a closed surface with boundary, such that $h|_{S_1}$ is pseudo-Anosov relative to a finite set. 
\end{theorem}

Hence, the only thing we have to prove is that in our case (\emph{i.e.} under the additional assumption that $\rho_{erg}(f)$ has nonempty interior in $H_1(S,\mathbb{R})$) we have $S_1=S$. 

\begin{proof}
Let $f\in \Homeo_0(S)$ be such that case 1. of Theorem~\ref{maintheorem} holds, \emph{i.e.} that the following does not hold: either $\rho_{erg}(f)$ is contained in an hyperplane of $H_1(S,\mathbb{R})$ or there exists two nonempty rational subspaces $E$ and $F$ of $H_1(S,\mathbb{R})$, which are orthogonal for the intersection form, and such that $\rho_{erg}(f) \subset E \cup F$. 

By Proposition \ref{Prop:decomposition}, the set $\rho_{erg}(f)$ has nonempty interior in $H_1(S,\mathbb{R})$, and a dense subset of $\mathrm{Conv} \left(\rho_{erg}(f) \right)$ belongs to $\rho_{erg}(f)$ and is realized by a periodic orbit. 

Let us build $x_1,\dots,x_k$ a finite set of $f$-periodic orbits, whose associated rotation vectors are $\rho_1,\dots,\rho_k\in H_1(S,\R)$, and suppose that these vectors span $H_1(S,\mathbb{R})$, and suppose that these vectors are not contained in any union $E\cup F$ of two linear subspaces $E,F\subset H_1(S,\mathbb{R})$ of positive codimension. To see that such a family does exist, consider a family $\tilde \rho_0=0, \tilde\rho_1,\dots,\tilde \rho_{2g}$ a family of $\rho_{erg}(f)$ that is a basis of $H_1(S,\mathbb{R})$ (it exists by hypothesis made on $f$). Now, for any partition $\mathcal P = \mathcal P_1\sqcup\mathcal P_2$ of $\{0,\dots,2g\}$, consider $\rho_{\mathcal P}\in \rho_{erg}(f)\setminus\big(\Span\{\tilde\rho_i\mid i\in\mathcal P_1\}\cup \Span\{\tilde\rho_j\mid j\in\mathcal P_2\}\big)$, that exists by the hypothesis made on $f$. Finally, consider a family $x_1,\dots,x_k$ of $f$-periodic orbits, such that the associated rotation vectors $\rho_1,\dots,\rho_k\in H_1(S,\R)$ are close to the vectors $\tilde \rho_0, \tilde\rho_1,\dots,\tilde \rho_{2g}, (\tilde\rho_{\mathcal P})_{\mathcal P \text{ partition}}$. If these two family are sufficiently close then $\rho_1,\dots,\rho_k$ span $H_1(S,\mathbb{R})$, and are not contained in any union $E\cup F$ of two linear subspaces $E,F\subset H_1(S,\mathbb{R})$ of positive codimension.\footnote{Of course there should be some much more economical way (in terms of the number $k$ of orbits) to build the family $x_1,\dots,x_k$, but we do not need any estimation on this number $k$ for our purpose.}

Now, set $F:=\bigcup_{n\in\N}\bigcup_{j=1}^k f^n(x_j)$ and $\Sigma := S\setminus F$. As the points $x_j$ are periodic, the set $F$ is finite. Apply Nielsen-Thurston classification (Theorem~\ref{TheoNTClass}) to $h:=f|_\Sigma$, and suppose there is at least one simple closed curve $C$ appearing in the decomposition \textit{(i)}. In the proof of \cite[Theorem 2]{MR1094554}, Pollicott shows that this implies that $C$ must be a separating curve; we can write $S = S_1\cup S_2$, with $S_1$ and $S_2$ being two surfaces with disjoint interiors and boundaries equal to $C$. But in this case, this implies that $\rho_{erg}(h)\subset H_1(S_1,\R)\cup H_1(S_2,\R)$ (in this equality, $h$ is seen as a homeomorphism of $S$, easily obtained by extending $h$ to $F$ by $f|_F$), which is impossible because there is some vector $\rho_i \in \rho_{erg}(h)$ with $\rho_i\notin H_1(S_1,\R)\cup H_1(S_2,\R)$. This shows that the set of curves appearing in Theorem~\ref{TheoNTClass} is empty, hence that $f$ is isotopic to a pseudo-Anosov relative to $F$. 
\end{proof}

Applying the arguments of \cite{zbMATH00009916} (that are based on Nielsen-Thurston theory) in the higher genus case, we get the following.

\begin{lemma}\label{LemNonemInte}
If $f\in\Homeo_0(S)$ is pseudo-Anosov relative to a finite subset of $S$, then any element of $\inte(\conv(\rho_{erg}(f)))$ is realized as the rotation vector of an ergodic measure. In particular, $\rho_{erg}(f)$ has nonempty interior.
\end{lemma} 

\begin{remark}
By a repeated use of \cite[Theorem F]{guiheneuf2022homotopic}, it is possible to replace the invariant finite set $F$ in the above proof by a single periodic orbit. Hence, any homeomorphism in $\mathrm{Homeo}_0(S)$ whose ergodic rotation set has nonempty interior is isotopic to a pseudo-Anosov homeomorphism relative to one periodic orbit.
\end{remark}

\section[Small rotation set homeomorphisms do not act hyperbolically]{Small rotation set homeomorphisms do not act hyperbolically on the fine curve graph}\label{Sec:Last}

Suppose that $1.$ does not hold in the statement of the theorem.

\paragraph{First case:} the rotation set $\rho_{erg}(f)$ is contained in a rational hyperplane $H$. 

By \cite{MR0425967} and \cite{MR0454995}, there exists a nonzero primitive integral vector $v_1 \in H^{\perp}$ and a simple nonseparating closed curve $c$ that represents $v_1$. The morphism
$$\begin{array}{rcl} 
H_1(S, \mathbb{Z}) & \longrightarrow & \mathbb{Z} \\
a & \longmapsto & a \wedge v_1
\end{array}$$
defines a covering map $S_c \rightarrow S$. Another way to see this covering map is the following (see Figure~\ref{FigFirstCover}). Cut the surface $S$ along $c$ to obtain a surface $S'$ with two boundary components $c_1$, $c_2$. Take a family $(S_i)_{i \in \mathbb{Z}}$ of surfaces, where each of the surfaces $S_i$ is a copy of the surface $S'$ and denote by $c_{i,1}$, and $c_{i,2}$ the boundary components of $S_i$ which correspond respectively to $c_1$ and $c_2$. To form the surface $S_c$, glue, for any $i$, the curve $c_{i,2}$ with the curve $c_{i+1,1}$ in such a way that corresponding points of $c$ are glued together. Using the map $S' \rightarrow S$, we obtain a covering map $S_c \rightarrow S$. With this second vision, we can define a map $\ell : S_c \rightarrow \mathbb{Z}$, which, to any point of $S_{i} \setminus S_{i+1}$, associates the value $i$.

\begin{figure}
\begin{center}

\tikzset{every picture/.style={line width=0.75pt}} 

\begin{tikzpicture}[x=0.75pt,y=0.75pt,yscale=-1.2,xscale=1.4]
\draw  [dash pattern={on 4.5pt off 4.5pt}]  (161.36,113.91) .. controls (169.91,115) and (172.82,115) .. (181.11,115) ;
\draw  [dash pattern={on 4.5pt off 4.5pt}]  (161.95,135) -- (181.11,135) ;
\draw  [dash pattern={on 4.5pt off 4.5pt}]  (481.11,135) -- (500.27,135) ;
\draw  [dash pattern={on 4.5pt off 4.5pt}]  (481.11,115) .. controls (488.45,115.18) and (490.45,115.18) .. (501.2,114.09) ;
\draw [color={rgb, 255:red, 0; green, 0; blue, 0 }  ,draw opacity=1 ][fill={rgb, 255:red, 0; green, 0; blue, 0 }  ,fill opacity=0.08 ]   (330.6,361.57) .. controls (310.4,360.97) and (300,370.97) .. (280.6,371.57) .. controls (232.8,372.57) and (232,302.17) .. (280.6,301.57) .. controls (292.62,301.46) and (311.57,311.6) .. (330.6,311.57) .. controls (360.4,310.57) and (370.4,301.77) .. (380.6,301.57) .. controls (428.45,302.14) and (431.6,372.57) .. (380.6,371.57) .. controls (370.8,371.77) and (352,361.77) .. (330.6,361.57) -- cycle ;
\draw [draw opacity=0][fill={rgb, 255:red, 255; green, 255; blue, 255 }  ,fill opacity=1 ]   (275.78,338.28) .. controls (283.52,329.61) and (291.18,332.19) .. (295.69,338.28) .. controls (289.35,342.53) and (282.02,342.86) .. (275.78,338.28) -- cycle ;
\draw    (270.6,334.23) .. controls (280.46,344.16) and (290.89,343.87) .. (300.6,334.23) ;
\draw    (275.78,338.28) .. controls (282.64,330.35) and (290.55,331.35) .. (295.69,338.28) ;

\draw [draw opacity=0][fill={rgb, 255:red, 255; green, 255; blue, 255 }  ,fill opacity=1 ]   (365.78,338.38) .. controls (373.52,329.71) and (381.18,332.29) .. (385.69,338.38) .. controls (379.35,342.63) and (372.02,342.96) .. (365.78,338.38) -- cycle ;
\draw    (360.6,334.33) .. controls (370.46,344.26) and (380.89,343.97) .. (390.6,334.33) ;
\draw    (365.78,338.38) .. controls (372.64,330.45) and (380.55,331.45) .. (385.69,338.38) ;

\draw [color={rgb, 255:red, 208; green, 2; blue, 27 }  ,draw opacity=1 ]   (284.46,341.64) .. controls (277.6,342) and (275.31,372.28) .. (283.6,371.14) ;
\draw [shift={(278.29,359.4)}, rotate = 273.71] [fill={rgb, 255:red, 208; green, 2; blue, 27 }  ,fill opacity=1 ][line width=0.08]  [draw opacity=0] (7.14,-3.43) -- (0,0) -- (7.14,3.43) -- (4.74,0) -- cycle    ;
\draw [color={rgb, 255:red, 208; green, 2; blue, 27 }  ,draw opacity=1 ] [dash pattern={on 4.5pt off 4.5pt}]  (284.46,341.64) .. controls (289.86,341.04) and (292.2,369.87) .. (283.6,371.14) ;
\draw    (330,265) -- (330,295.75) ;
\draw [shift={(330,298.75)}, rotate = 270] [fill={rgb, 255:red, 0; green, 0; blue, 0 }  ][line width=0.08]  [draw opacity=0] (8.93,-4.29) -- (0,0) -- (8.93,4.29) -- (5.93,0) -- cycle    ;
\draw  [fill={rgb, 255:red, 0; green, 0; blue, 0 }  ,fill opacity=0.08 ] (278.89,229.97) .. controls (358.69,230.07) and (269.09,169.87) .. (328.89,169.97) .. controls (389.29,169.87) and (299.09,230.27) .. (378.89,229.97) .. controls (384.14,230.03) and (384.02,250.03) .. (378.89,249.97) .. controls (368.89,250.07) and (288.69,249.87) .. (278.89,249.97) .. controls (273.89,250.16) and (273.89,230.03) .. (278.89,229.97) -- cycle ;
\draw [draw opacity=0][fill={rgb, 255:red, 255; green, 255; blue, 255 }  ,fill opacity=1 ]   (319.07,190.55) .. controls (326.81,181.88) and (334.47,184.47) .. (338.98,190.55) .. controls (332.64,194.8) and (325.31,195.13) .. (319.07,190.55) -- cycle ;
\draw    (313.89,186.5) .. controls (323.75,196.43) and (334.18,196.15) .. (343.89,186.5) ;
\draw    (319.07,190.55) .. controls (325.93,182.62) and (333.84,183.62) .. (338.98,190.55) ;

\draw  [color={rgb, 255:red, 208; green, 2; blue, 27 }  ,draw opacity=1 ] (378.89,249.97) .. controls (373.64,249.91) and (373.77,229.78) .. (378.89,229.97) .. controls (384.02,230.16) and (384.14,250.03) .. (378.89,249.97) -- cycle ;
\draw  [color={rgb, 255:red, 208; green, 2; blue, 27 }  ,draw opacity=1 ] (278.89,249.97) .. controls (273.64,249.91) and (273.77,229.78) .. (278.89,229.97) .. controls (284.02,230.16) and (284.14,250.03) .. (278.89,249.97) -- cycle ;

\draw  [fill={rgb, 255:red, 0; green, 0; blue, 0 }  ,fill opacity=0.08 ] (393.89,229.97) .. controls (473.69,230.07) and (384.09,169.87) .. (443.89,169.97) .. controls (504.29,169.87) and (414.09,230.27) .. (493.89,229.97) .. controls (499.14,230.03) and (499.02,250.03) .. (493.89,249.97) .. controls (483.89,250.07) and (403.69,249.87) .. (393.89,249.97) .. controls (388.89,250.16) and (388.89,230.03) .. (393.89,229.97) -- cycle ;
\draw [draw opacity=0][fill={rgb, 255:red, 255; green, 255; blue, 255 }  ,fill opacity=1 ]   (434.07,190.55) .. controls (441.81,181.88) and (449.47,184.47) .. (453.98,190.55) .. controls (447.64,194.8) and (440.31,195.13) .. (434.07,190.55) -- cycle ;
\draw    (428.89,186.5) .. controls (438.75,196.43) and (449.18,196.15) .. (458.89,186.5) ;
\draw    (434.07,190.55) .. controls (440.93,182.62) and (448.84,183.62) .. (453.98,190.55) ;

\draw  [color={rgb, 255:red, 208; green, 2; blue, 27 }  ,draw opacity=1 ] (493.89,249.97) .. controls (488.64,249.91) and (488.77,229.78) .. (493.89,229.97) .. controls (499.02,230.16) and (499.14,250.03) .. (493.89,249.97) -- cycle ;
\draw  [color={rgb, 255:red, 208; green, 2; blue, 27 }  ,draw opacity=1 ] (393.89,249.97) .. controls (388.64,249.91) and (388.77,229.78) .. (393.89,229.97) .. controls (399.02,230.16) and (399.14,250.03) .. (393.89,249.97) -- cycle ;

\draw  [fill={rgb, 255:red, 0; green, 0; blue, 0 }  ,fill opacity=0.08 ] (166.11,229.97) .. controls (245.91,230.07) and (156.31,169.87) .. (216.11,169.97) .. controls (276.51,169.87) and (186.31,230.27) .. (266.11,229.97) .. controls (271.36,230.03) and (271.23,250.03) .. (266.11,249.97) .. controls (256.11,250.07) and (175.91,249.87) .. (166.11,249.97) .. controls (161.11,250.16) and (161.11,230.03) .. (166.11,229.97) -- cycle ;
\draw [draw opacity=0][fill={rgb, 255:red, 255; green, 255; blue, 255 }  ,fill opacity=1 ]   (206.29,190.55) .. controls (214.03,181.88) and (221.69,184.47) .. (226.2,190.55) .. controls (219.86,194.8) and (212.53,195.13) .. (206.29,190.55) -- cycle ;
\draw    (201.11,186.5) .. controls (210.97,196.43) and (221.39,196.15) .. (231.11,186.5) ;
\draw    (206.29,190.55) .. controls (213.15,182.62) and (221.06,183.62) .. (226.2,190.55) ;

\draw  [color={rgb, 255:red, 208; green, 2; blue, 27 }  ,draw opacity=1 ] (266.11,249.97) .. controls (260.86,249.91) and (260.98,229.78) .. (266.11,229.97) .. controls (271.23,230.16) and (271.36,250.03) .. (266.11,249.97) -- cycle ;
\draw  [color={rgb, 255:red, 208; green, 2; blue, 27 }  ,draw opacity=1 ] (166.11,249.97) .. controls (160.86,249.91) and (160.98,229.78) .. (166.11,229.97) .. controls (171.23,230.16) and (171.36,250.03) .. (166.11,249.97) -- cycle ;

\draw  [fill={rgb, 255:red, 0; green, 0; blue, 0 }  ,fill opacity=0.08 ] (281.11,115) .. controls (360.91,115.1) and (271.31,54.9) .. (331.11,55) .. controls (391.51,54.9) and (301.31,115.3) .. (381.11,115) .. controls (386.36,115.06) and (386.23,135.06) .. (381.11,135) .. controls (371.11,135.1) and (290.91,134.9) .. (281.11,135) .. controls (276.11,135.19) and (276.11,115.06) .. (281.11,115) -- cycle ;
\draw [draw opacity=0][fill={rgb, 255:red, 255; green, 255; blue, 255 }  ,fill opacity=1 ]   (321.29,75.58) .. controls (329.03,66.91) and (336.69,69.5) .. (341.2,75.58) .. controls (334.86,79.83) and (327.53,80.16) .. (321.29,75.58) -- cycle ;
\draw    (316.11,71.53) .. controls (325.97,81.46) and (336.39,81.18) .. (346.11,71.53) ;
\draw    (321.29,75.58) .. controls (328.15,67.65) and (336.06,68.65) .. (341.2,75.58) ;

\draw  [color={rgb, 255:red, 208; green, 2; blue, 27 }  ,draw opacity=1 ] (381.11,135) .. controls (375.86,134.94) and (375.98,114.81) .. (381.11,115) .. controls (386.23,115.19) and (386.36,135.06) .. (381.11,135) -- cycle ;
\draw  [color={rgb, 255:red, 208; green, 2; blue, 27 }  ,draw opacity=1 ] (281.11,135) .. controls (275.86,134.94) and (275.98,114.81) .. (281.11,115) .. controls (286.23,115.19) and (286.36,135.06) .. (281.11,135) -- cycle ;

\draw  [fill={rgb, 255:red, 0; green, 0; blue, 0 }  ,fill opacity=0.08 ] (181.11,115) .. controls (260.91,115.1) and (171.31,54.9) .. (231.11,55) .. controls (291.51,54.9) and (201.31,115.3) .. (281.11,115) .. controls (286.36,115.06) and (286.23,135.06) .. (281.11,135) .. controls (271.11,135.1) and (190.91,134.9) .. (181.11,135) .. controls (176.11,135.19) and (176.11,115.06) .. (181.11,115) -- cycle ;
\draw [draw opacity=0][fill={rgb, 255:red, 255; green, 255; blue, 255 }  ,fill opacity=1 ]   (221.29,75.58) .. controls (229.03,66.91) and (236.69,69.5) .. (241.2,75.58) .. controls (234.86,79.83) and (227.53,80.16) .. (221.29,75.58) -- cycle ;
\draw    (216.11,71.53) .. controls (225.97,81.46) and (236.39,81.18) .. (246.11,71.53) ;
\draw    (221.29,75.58) .. controls (228.15,67.65) and (236.06,68.65) .. (241.2,75.58) ;

\draw  [color={rgb, 255:red, 208; green, 2; blue, 27 }  ,draw opacity=1 ] (281.11,135) .. controls (275.86,134.94) and (275.98,114.81) .. (281.11,115) .. controls (286.23,115.19) and (286.36,135.06) .. (281.11,135) -- cycle ;
\draw  [color={rgb, 255:red, 208; green, 2; blue, 27 }  ,draw opacity=1 ] (181.11,135) .. controls (175.86,134.94) and (175.98,114.81) .. (181.11,115) .. controls (186.23,115.19) and (186.36,135.06) .. (181.11,135) -- cycle ;

\draw  [fill={rgb, 255:red, 0; green, 0; blue, 0 }  ,fill opacity=0.08 ] (381.11,115) .. controls (460.91,115.1) and (371.31,54.9) .. (431.11,55) .. controls (491.51,54.9) and (401.31,115.3) .. (481.11,115) .. controls (486.36,115.06) and (486.23,135.06) .. (481.11,135) .. controls (471.11,135.1) and (390.91,134.9) .. (381.11,135) .. controls (376.11,135.19) and (376.11,115.06) .. (381.11,115) -- cycle ;
\draw [draw opacity=0][fill={rgb, 255:red, 255; green, 255; blue, 255 }  ,fill opacity=1 ]   (421.29,75.58) .. controls (429.03,66.91) and (436.69,69.5) .. (441.2,75.58) .. controls (434.86,79.83) and (427.53,80.16) .. (421.29,75.58) -- cycle ;
\draw    (416.11,71.53) .. controls (425.97,81.46) and (436.39,81.18) .. (446.11,71.53) ;
\draw    (421.29,75.58) .. controls (428.15,67.65) and (436.06,68.65) .. (441.2,75.58) ;

\draw  [color={rgb, 255:red, 208; green, 2; blue, 27 }  ,draw opacity=1 ] (481.11,135) .. controls (475.86,134.94) and (475.98,114.81) .. (481.11,115) .. controls (486.23,115.19) and (486.36,135.06) .. (481.11,135) -- cycle ;
\draw  [color={rgb, 255:red, 208; green, 2; blue, 27 }  ,draw opacity=1 ] (381.11,135) .. controls (375.86,134.94) and (375.98,114.81) .. (381.11,115) .. controls (386.23,115.19) and (386.36,135.06) .. (381.11,135) -- cycle ;

\draw    (180,140) .. controls (214.13,145.01) and (239.15,145.41) .. (277.06,140.4) ;
\draw [shift={(280,140)}, rotate = 172.19] [fill={rgb, 255:red, 0; green, 0; blue, 0 }  ][line width=0.08]  [draw opacity=0] (8.04,-3.86) -- (0,0) -- (8.04,3.86) -- (5.34,0) -- cycle    ;

\draw (275,356) node [anchor=east] [inner sep=0.75pt]  [color={rgb, 255:red, 208; green, 2; blue, 27 }  ,opacity=1 ]  {$c$};
\draw (347,306.07) node [anchor=south east] [inner sep=0.75pt]    {$S$};
\draw (300,252) node [anchor=north west][inner sep=0.75pt]    {$S_{0}$};
\draw (436,255.4) node [anchor=north west][inner sep=0.75pt]    {$S_{1}$};
\draw (201,255.4) node [anchor=north west][inner sep=0.75pt]    {$S_{-1}$};
\draw (283.59,239.91) node [anchor=west] [inner sep=0.75pt]  [color={rgb, 255:red, 208; green, 2; blue, 27 }  ,opacity=1 ]  {$c_{0,1}$};
\draw (374.86,239.8) node [anchor=east] [inner sep=0.75pt]  [color={rgb, 255:red, 208; green, 2; blue, 27 }  ,opacity=1 ]  {$c_{0,2}$};
\draw (487.89,240) node [anchor=east] [inner sep=0.75pt]  [color={rgb, 255:red, 208; green, 2; blue, 27 }  ,opacity=1 ]  {$c_{1,2}$};
\draw (261,240.14) node [anchor=east] [inner sep=0.75pt]  [color={rgb, 255:red, 208; green, 2; blue, 27 }  ,opacity=1 ]  {$c_{-1,2}$};
\draw (398.45,240.49) node [anchor=west] [inner sep=0.75pt]  [color={rgb, 255:red, 208; green, 2; blue, 27 }  ,opacity=1 ]  {$c_{1,1}$};
\draw (171.59,239.8) node [anchor=west] [inner sep=0.75pt]  [color={rgb, 255:red, 208; green, 2; blue, 27 }  ,opacity=1 ]  {$c_{-1,1}$};
\draw (400.57,137.03) node [anchor=north west][inner sep=0.75pt]    {$S_{c}$};
\draw (228.23,147.4) node [anchor=north] [inner sep=0.75pt]    {$T$};
\draw (329.59,152.47) node    {$\simeq $};

\end{tikzpicture}

\caption{\label{FigFirstCover} The covering map $S_c\to S$ of the first case.}
\end{center}
\end{figure}
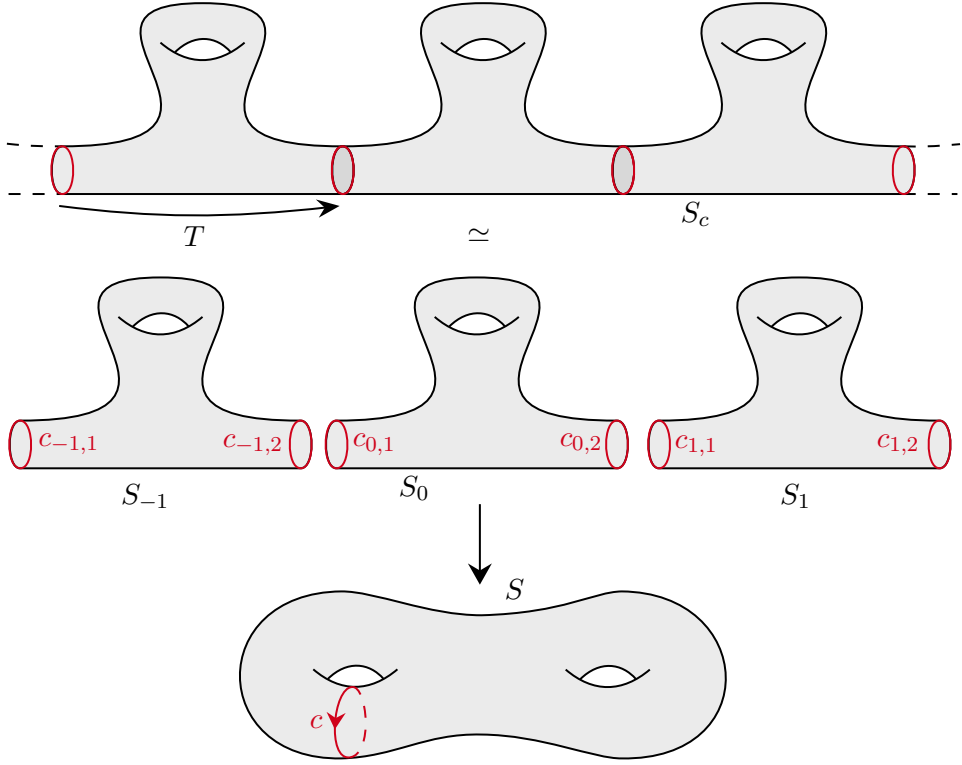
  
The group of automorphisms of this covering map is infinite cyclic and is generated by an automorphism which we will denote by $T$. We choose $T$ in such a way that, for any $i$, $T$ sends $S_i$ to $S_{i+1}$
  
Fix also a lift $f_c:S_c \rightarrow S_c$ of $f$. This lift commutes with $T$. Using this lift, we can define as follows new rotation sets for $f$, which are adapted to this lift.

We define $\rho^{t}_{v_1}(f)$ as the subset of $\mathbb{R}$ consisting of limit values of sequences of the form
$$\left( \frac{\ell(f_c^n(x_n))-\ell(x_n)}{n} \right),$$
where $(x_n)_n$ is any sequence of points of $S_c$. For any point $x$ of $S$, choose a lift $x_c$ of $x$ in such a way that $x \mapsto x_c$ is a measurable map. Let us denote by $\mathcal{M}(f)$ the set of $f$-invariant Borel probability measures on $S$ and by $\mathcal{M}_{erg}(f)$ the subset of $\mathcal{M}(f)$ consisting of ergodic probability measures.

\begin{lemma} \label{Lem:transverserotationset}
\begin{align*}
\rho^{t}_{v_1}(f) & = \left\{ \int_{S} \ell(f_c(x_c))-\ell(x_c) \ud\mu(x) \ | \ \mu \in \mathcal{M}(f) \right\} \\
& = \mathrm{Conv}\left( \left\{ \int_{S} \ell(f_c(x_c))-\ell(x_c) \ud\mu(x) \ | \ \mu \in \mathcal{M}_{erg}(f) \right\} \right).
\end{align*}
\end{lemma} 

\begin{proof}
The equality
\begin{multline*}
\left\{ \int_{S} \ell(f_c(x_c))-\ell(x_c) \ud\mu(x) \ | \ \mu \in \mathcal{M}(f) \right\} \\
=  \mathrm{Conv} \left( \left\{ \int_{S} \ell(f_c(x_c))-\ell(x_c) \ud\mu(x) \ | \ \mu \in \mathcal{M}_{erg}(f) \right\} \right)
\end{multline*}
holds as $\mathcal{M}(f)$ is the convex hull of $\mathcal{M}_{erg}(f)$, and as the map $\mu\mapsto \int_{S} \ell(f_c(x_c))-\ell(x_c) \ud\mu(x)$ is affine.

Take a real number $r \in \rho^{t}_{v_1}(f)$. Then there exists a sequence of integers $(n_k)_k$ and a sequence $(x'_k)_k$ of points of $S_c$ such that
$$ \lim_{k \rightarrow + \infty} \frac{\ell(f_c^{n_k}(x'_k))-\ell(x'_k)}{n_k}=r.$$
For any $k$, let $x_k$ be the projection of $x'_k$ on the surface $S$ and denote
$$\nu_k = \frac{1}{n_k} \sum_{k=0}^{n_k-1} \delta_{f^{k}(x_k)}.$$
Then
$$\frac{\ell(f_c^{n_k}(x'_k))-\ell(x'_k)}{n_k}=\int_{S} \ell(f_c(x_c))-\ell(x_c) d\nu_k(x).$$
Let $\mu$ be a weak-* limit of the probability measures $\nu_k$. Observe that $\mu \in \mathcal{M}(f)$ and that
$$r=\int_{S} \ell(f_c(x_c))-\ell(x_c) \ud\mu(x).$$
This shows the inclusion $\subset$ for the first equality of the lemma.

Finally, if $\mu \in \mathcal{M}_{erg}(f)$, apply Birkhoff ergodic theorem to the function $x \mapsto \ell(f_c(x_c))-\ell(x_c)$ to obtain that
$$\int_{S} \ell(f_c(x_c))-\ell(x_c) \ud\mu(x) \in \rho^{t}_{v_1}(f).$$
This shows the reverse inclusion for the first equality of the lemma.
\end{proof}

As a consequence of the Lemma \ref{Lem:transverserotationset} and Lemma \ref{Lem:extremalpoints}, we obtain the following result.

\begin{lemma} \label{Lem:rotationsetcomparison}
\begin{align*}
\rho^{t}_{v_1}(f) & = \left\{ a \wedge v_1 \mid a \in \rho_{H_1}(f) \right\} \\
& = \mathrm{Conv}\big( \{ a \wedge v_1 \mid a \in \rho_{erg}(f) \} \big).
\end{align*}
\end{lemma} 

\begin{proof}
Take any differentiable map $\lambda:S_c \rightarrow \mathbb{R}$ such that, for any point $z$ on $S_c$, $\lambda(T(z))=\lambda(z)+1$. The differential of this map is invariant under $T$ and defines a closed $1$-form $\omega$ on the surface $S$. Observe that, for any closed curve $\gamma:[0,1] \rightarrow S$,
$$\int_{\gamma} \omega= \ell(\gamma_c(1))- \ell(\gamma_c(0))= [\gamma] \wedge v_1 ,$$
where $\gamma_c$ is a lift of $\gamma$ to the surface $S_c$.

Then, by Lemma \ref{Lem:extremalpoints},
$$d(\rho_{H_1}(f))(\omega)=\left\{ a \wedge v_1 \ | \ a \in \rho_{H_1}(f) \right\}  = \mathrm{Conv}\left( \left\{ a \wedge v_1 \ | \ a \in \rho_{erg}(f) \right\} \right)$$
as $d(\rho_{H_1}(f))(\omega)$ is a connected and hence a convex subset of $\mathbb{R}$. 

It remains to prove that $d(\rho_{H_1}(f))(\omega)=\rho^{t}_{v_1}(f)$. This last equality is a consequence of the fact that any converging sequence of the form $(\frac{1}{n_k}d(c_{x_k,n_k})(\omega))_k$, where $n_k \rightarrow +\infty$ and $(x_k)_k$ is a sequence of points on the surface $S$, has the same limit as the sequence $(\frac{1}{n_k}(\ell(f_c^{n_k}(x_{k,c}))- \ell(x_{k,c})))_k$.
\end{proof}

By Lemma \ref{Lem:rotationsetcomparison}, if $\rho_{erg}(f)$ is contained in $v_1^{\perp}$, then $\rho^{t}_{v_1}(f)=\left\{ 0 \right\}$.

We now prove that, in this case, $f$ does not act hyperbolically on the fine curve graph of $S$.

For any simple closed curve $\alpha$ that is isotopic to $c$, we define $\mathcal{C}_c(\alpha)$ as the number of lifts of $c$ to $S_c$ that a given lift of $\alpha$ to $S_c$ meets. This quantity does not depend on the chosen lift of $\alpha$.

The following lemma is analogous to Lemma 4.5 in \cite{BHMMW}.

\begin{lemma} \label{Lem:distanceestimates}
For any simple closed curve $\alpha$ on $S$ which is isotopic to $c$,
$$\mathcal{C}_c(\alpha)+1 \geq d_{C^{\dagger}}(\alpha,c).$$
\end{lemma}

Before proving the lemma, let us see why it allows us to complete the proof in this case. As $\rho^{t}_{v_1}(f)=\left\{ 0 \right\}$, 
$$ \lim_{n \rightarrow +\infty} \frac{1}{n} \mathcal{C}_c(f^{n}(c)) =0.$$
Indeed, otherwise, we would find a sequence of integers $n_k \rightarrow +\infty$ and a sequence of points $(x_k)$ of $c$ such that 
$$\lim_{k \rightarrow +\infty} \frac{1}{n_k}\big(\ell(f_c^{n_k}(x_{k,c}))- \ell(x_{k,c})\big) \neq 0,$$
a contradiction with $\rho^{t}_{v_1}(f)=\left\{ 0 \right\}$.
Lemma \ref{Lem:distanceestimates} implies then that
$$\lim_{n \rightarrow +\infty} \frac{1}{n}d_{C^{\dagger}}(f^{n}(c),c)=0$$
so that $f$ does not act hyperbolically on $C^{\dagger}(S)$.

\begin{proof}[Proof of Lemma \ref{Lem:distanceestimates}]
The proof is similar to the proof of \cite[Lemma 4.5]{BHMMW}. We prove it by induction on $n=\mathcal{C}_c(\alpha)$. 

For $n=0$, the curves $\alpha$ and $c$ have to be disjoint so that $d_{C^{\dagger}}(\alpha,c)=1$.

Suppose that the inequality holds for any curve $\alpha$ with $\mathcal{C}_c(\alpha)=n$.
Fix now a curve $\alpha$ such that $\mathcal{C}_c(\alpha)=n+1$. We will find a curve $\beta$ such that
$d_{C^{\dagger}}(\alpha,\beta)=1$ and $\mathcal{C}_c(\beta)=n$, which will complete the induction.

\begin{figure}
\begin{center}

\tikzset{every picture/.style={line width=0.75pt}} 

\begin{tikzpicture}[x=0.75pt,y=0.75pt,yscale=-1.2,xscale=1.25]

\draw  [fill={rgb, 255:red, 0; green, 0; blue, 0 }  ,fill opacity=0.08 ] (501.19,134.97) .. controls (580.99,135.07) and (491.39,74.87) .. (551.19,74.97) .. controls (611.59,74.87) and (521.39,135.27) .. (601.19,134.97) .. controls (606.44,135.04) and (606.32,155.04) .. (601.19,154.97) .. controls (591.19,155.07) and (510.99,154.87) .. (501.19,154.97) .. controls (496.19,155.16) and (496.19,135.04) .. (501.19,134.97) -- cycle ;
\draw [draw opacity=0][fill={rgb, 255:red, 255; green, 255; blue, 255 }  ,fill opacity=1 ]   (541.37,95.55) .. controls (549.11,86.89) and (556.78,89.47) .. (561.28,95.55) .. controls (554.94,99.8) and (547.61,100.14) .. (541.37,95.55) -- cycle ;
\draw    (536.19,91.51) .. controls (546.05,101.44) and (556.48,101.15) .. (566.19,91.51) ;
\draw    (541.37,95.55) .. controls (548.23,87.62) and (556.14,88.62) .. (561.28,95.55) ;

\draw  [color={rgb, 255:red, 208; green, 2; blue, 27 }  ,draw opacity=1 ] (601.19,154.97) .. controls (595.94,154.91) and (596.07,134.79) .. (601.19,134.97) .. controls (606.32,135.16) and (606.44,155.04) .. (601.19,154.97) -- cycle ;
\draw  [color={rgb, 255:red, 208; green, 2; blue, 27 }  ,draw opacity=1 ] (501.19,154.97) .. controls (495.94,154.91) and (496.07,134.79) .. (501.19,134.97) .. controls (506.32,135.16) and (506.44,155.04) .. (501.19,154.97) -- cycle ;

\draw  [dash pattern={on 4.5pt off 4.5pt}]  (181.36,133.91) .. controls (189.91,135) and (192.82,135) .. (201.11,135) ;
\draw  [dash pattern={on 4.5pt off 4.5pt}]  (181.95,155) -- (201.11,155) ;
\draw  [dash pattern={on 4.5pt off 4.5pt}]  (501.11,155) -- (520.27,155) ;
\draw  [fill={rgb, 255:red, 0; green, 0; blue, 0 }  ,fill opacity=0.08 ] (201.11,135) .. controls (280.91,135.1) and (191.31,74.9) .. (251.11,75) .. controls (311.51,74.9) and (221.31,135.3) .. (301.11,135) .. controls (306.36,135.06) and (306.23,155.06) .. (301.11,155) .. controls (291.11,155.1) and (210.91,154.9) .. (201.11,155) .. controls (196.11,155.19) and (196.11,135.06) .. (201.11,135) -- cycle ;
\draw [draw opacity=0][fill={rgb, 255:red, 255; green, 255; blue, 255 }  ,fill opacity=1 ]   (241.29,95.58) .. controls (249.03,86.91) and (256.69,89.5) .. (261.2,95.58) .. controls (254.86,99.83) and (247.53,100.16) .. (241.29,95.58) -- cycle ;
\draw    (236.11,91.53) .. controls (245.97,101.46) and (256.39,101.18) .. (266.11,91.53) ;
\draw    (241.29,95.58) .. controls (248.15,87.65) and (256.06,88.65) .. (261.2,95.58) ;

\draw  [color={rgb, 255:red, 208; green, 2; blue, 27 }  ,draw opacity=1 ] (301.11,155) .. controls (295.86,154.94) and (295.98,134.81) .. (301.11,135) .. controls (306.23,135.19) and (306.36,155.06) .. (301.11,155) -- cycle ;
\draw  [color={rgb, 255:red, 208; green, 2; blue, 27 }  ,draw opacity=1 ] (201.11,155) .. controls (195.86,154.94) and (195.98,134.81) .. (201.11,135) .. controls (206.23,135.19) and (206.36,155.06) .. (201.11,155) -- cycle ;

\draw  [fill={rgb, 255:red, 74; green, 144; blue, 226 }  ,fill opacity=0.3 ] (401.11,135) .. controls (480.91,135.1) and (391.31,74.9) .. (451.11,75) .. controls (511.51,74.9) and (421.31,135.3) .. (501.11,135) .. controls (506.36,135.06) and (506.23,155.06) .. (501.11,155) .. controls (491.11,155.1) and (410.91,154.9) .. (401.11,155) .. controls (396.11,155.19) and (396.11,135.06) .. (401.11,135) -- cycle ;
\draw [draw opacity=0][fill={rgb, 255:red, 255; green, 255; blue, 255 }  ,fill opacity=1 ]   (441.29,95.58) .. controls (449.03,86.91) and (456.69,89.5) .. (461.2,95.58) .. controls (454.86,99.83) and (447.53,100.16) .. (441.29,95.58) -- cycle ;
\draw    (436.11,91.53) .. controls (445.97,101.46) and (456.39,101.18) .. (466.11,91.53) ;
\draw    (441.29,95.58) .. controls (448.15,87.65) and (456.06,88.65) .. (461.2,95.58) ;

\draw  [color={rgb, 255:red, 208; green, 2; blue, 27 }  ,draw opacity=1 ] (401.11,155) .. controls (395.86,154.94) and (395.98,134.81) .. (401.11,135) .. controls (406.23,135.19) and (406.36,155.06) .. (401.11,155) -- cycle ;
\draw    (200,160) .. controls (234.13,165.01) and (259.15,165.41) .. (297.06,160.4) ;
\draw [shift={(300,160)}, rotate = 172.19] [fill={rgb, 255:red, 0; green, 0; blue, 0 }  ][line width=0.08]  [draw opacity=0] (8.04,-3.86) -- (0,0) -- (8.04,3.86) -- (5.34,0) -- cycle    ;
\draw [color={rgb, 255:red, 65; green, 117; blue, 5 }  ,draw opacity=1 ] [dash pattern={on 4.5pt off 4.5pt}]  (280.11,132.94) .. controls (274.38,131.15) and (275.31,154.69) .. (280.11,154.94) ;
\draw  [dash pattern={on 4.5pt off 4.5pt}]  (601.19,154.97) -- (620.36,154.97) ;
\draw  [dash pattern={on 4.5pt off 4.5pt}]  (601.19,134.97) .. controls (608.54,135.16) and (610.54,135.16) .. (621.28,134.07) ;
\draw [color={rgb, 255:red, 65; green, 117; blue, 5 }  ,draw opacity=1 ] [dash pattern={on 4.5pt off 4.5pt}]  (480.71,133.14) .. controls (474.98,131.35) and (475.91,154.89) .. (480.71,155.14) ;
\draw  [draw opacity=0][fill={rgb, 255:red, 255; green, 255; blue, 255 }  ,fill opacity=1 ] (401.11,135) .. controls (395.89,135.06) and (396.17,154.92) .. (401.11,155) .. controls (403.42,154.96) and (404.81,150.4) .. (404.92,147.63) .. controls (414,146.88) and (413.42,142.96) .. (404.83,141.04) .. controls (404.17,138.15) and (402.92,134.79) .. (401.11,135) -- cycle ;
\draw  [draw opacity=0][fill={rgb, 255:red, 0; green, 0; blue, 0 }  ,fill opacity=0.08 ] (404.83,141.04) .. controls (412.17,143.29) and (415.17,145.88) .. (404.92,147.63) .. controls (405.08,145.79) and (405.08,144.04) .. (404.83,141.04) -- cycle ;
\draw  [fill={rgb, 255:red, 0; green, 0; blue, 0 }  ,fill opacity=0.08 ] (301.11,135) .. controls (380.91,135.1) and (291.31,74.9) .. (351.11,75) .. controls (411.51,74.9) and (321.31,135.3) .. (401.11,135) .. controls (406.36,135.06) and (406.23,155.06) .. (401.11,155) .. controls (391.11,155.1) and (310.91,154.9) .. (301.11,155) .. controls (296.11,155.19) and (296.11,135.06) .. (301.11,135) -- cycle ;
\draw [draw opacity=0][fill={rgb, 255:red, 255; green, 255; blue, 255 }  ,fill opacity=1 ]   (341.29,95.58) .. controls (349.03,86.91) and (356.69,89.5) .. (361.2,95.58) .. controls (354.86,99.83) and (347.53,100.16) .. (341.29,95.58) -- cycle ;
\draw    (336.11,91.53) .. controls (345.97,101.46) and (356.39,101.18) .. (366.11,91.53) ;
\draw    (341.29,95.58) .. controls (348.15,87.65) and (356.06,88.65) .. (361.2,95.58) ;

\draw  [color={rgb, 255:red, 208; green, 2; blue, 27 }  ,draw opacity=1 ] (401.11,155) .. controls (395.86,154.94) and (395.98,134.81) .. (401.11,135) .. controls (406.23,135.19) and (406.36,155.06) .. (401.11,155) -- cycle ;
\draw  [color={rgb, 255:red, 208; green, 2; blue, 27 }  ,draw opacity=1 ] (301.11,155) .. controls (295.86,154.94) and (295.98,134.81) .. (301.11,135) .. controls (306.23,135.19) and (306.36,155.06) .. (301.11,155) -- cycle ;

\draw [color={rgb, 255:red, 65; green, 117; blue, 5 }  ,draw opacity=1 ]   (280.11,132.94) .. controls (291.15,137.15) and (328.3,150.68) .. (337.11,127.61) .. controls (345.92,104.54) and (323.15,92.99) .. (336,81.61) .. controls (348.85,70.23) and (373.46,79.92) .. (370,91.5) .. controls (366.54,103.08) and (356.08,133) .. (370.22,136.39) .. controls (384.37,139.78) and (412.03,139.1) .. (411.46,145) .. controls (410.9,150.9) and (373.31,146.48) .. (362,143.13) .. controls (350.69,139.77) and (364.49,100.46) .. (367.67,91.5) .. controls (370.85,82.54) and (339.5,80.91) .. (335.56,86.72) .. controls (331.62,92.54) and (339.92,107.31) .. (346.23,124.85) .. controls (352.54,142.38) and (288.54,154.85) .. (280.11,154.94) ;
\draw  [draw opacity=0][fill={rgb, 255:red, 255; green, 255; blue, 255 }  ,fill opacity=1 ] (480.71,133.14) .. controls (487.67,134.33) and (493.19,134.9) .. (501.19,134.97) .. controls (502.75,135.13) and (504,137.04) .. (504.5,139.96) .. controls (498.25,139.04) and (492.05,137.1) .. (480.71,133.14) -- cycle ;
\draw  [fill={rgb, 255:red, 0; green, 0; blue, 0 }  ,fill opacity=0.08 ] (480.63,133.17) .. controls (487.58,134.36) and (493.11,134.93) .. (501.11,135) .. controls (502.67,135.15) and (503.92,137.07) .. (504.42,139.98) .. controls (498.17,139.07) and (491.96,137.12) .. (480.63,133.17) -- cycle ;
\draw  [draw opacity=0][fill={rgb, 255:red, 255; green, 255; blue, 255 }  ,fill opacity=1 ] (480.71,155.14) .. controls (489.17,154.46) and (492.58,153.29) .. (504.5,150.96) .. controls (503.92,152.63) and (503,155.13) .. (501.11,155) .. controls (492.92,154.96) and (489.58,154.96) .. (480.71,155.14) -- cycle ;
\draw  [fill={rgb, 255:red, 0; green, 0; blue, 0 }  ,fill opacity=0.08 ] (480.71,155.14) .. controls (489.17,154.46) and (492.58,153.29) .. (504.5,150.96) .. controls (503.92,152.63) and (503,155.13) .. (501.11,155) .. controls (492.92,154.96) and (489.58,154.96) .. (480.71,155.14) -- cycle ;
\draw  [color={rgb, 255:red, 208; green, 2; blue, 27 }  ,draw opacity=1 ] (501.11,155) .. controls (495.86,154.94) and (495.98,134.81) .. (501.11,135) .. controls (506.23,135.19) and (506.36,155.06) .. (501.11,155) -- cycle ;
\draw [color={rgb, 255:red, 65; green, 117; blue, 5 }  ,draw opacity=1 ]   (480.71,133.14) .. controls (491.75,137.35) and (528.9,150.88) .. (537.71,127.81) .. controls (546.52,104.74) and (523.75,93.19) .. (536.6,81.81) .. controls (549.45,70.43) and (574.06,80.12) .. (570.6,91.7) .. controls (567.14,103.28) and (556.68,133.2) .. (570.82,136.59) .. controls (584.97,139.98) and (612.63,139.3) .. (612.06,145.2) .. controls (611.5,151.1) and (573.91,146.68) .. (562.6,143.32) .. controls (551.29,139.97) and (565.09,100.66) .. (568.27,91.7) .. controls (571.45,82.74) and (540.1,81.11) .. (536.16,86.92) .. controls (532.22,92.74) and (540.52,107.51) .. (546.83,125.05) .. controls (553.14,142.58) and (489.14,155.05) .. (480.71,155.14) ;
\draw [color={rgb, 255:red, 189; green, 16; blue, 224 }  ,draw opacity=1 ]   (468.52,126.24) .. controls (463.32,119.04) and (456.7,154.84) .. (462.81,155) ;
\draw [color={rgb, 255:red, 189; green, 16; blue, 224 }  ,draw opacity=1 ] [dash pattern={on 4.5pt off 4.5pt}]  (468.52,126.24) .. controls (471.1,128.9) and (465.38,155) .. (462.81,155) ;

\draw (420.57,157.03) node [anchor=north west][inner sep=0.75pt]    {$S_{c}$};
\draw (248.23,167.4) node [anchor=north] [inner sep=0.75pt]    {$T$};
\draw (427.86,131.95) node [anchor=north west][inner sep=0.75pt]  [color={rgb, 255:red, 0; green, 80; blue, 166 }  ,opacity=1 ]  {$\Sigma '$};
\draw (401.43,132.98) node [anchor=south] [inner sep=0.75pt]  [color={rgb, 255:red, 208; green, 2; blue, 27 }  ,opacity=1 ]  {$\alpha _{c}$};
\draw (301.11,131.6) node [anchor=south] [inner sep=0.75pt]  [color={rgb, 255:red, 208; green, 2; blue, 27 }  ,opacity=1 ]  {$T^{-1} \alpha _{c}$};
\draw (342,135) node [anchor=north west][inner sep=0.75pt]  [color={rgb, 255:red, 65; green, 117; blue, 5 }  ,opacity=1 ]  {$c_{c}$};
\draw (550.59,138) node [anchor=north] [inner sep=0.75pt]  [color={rgb, 255:red, 65; green, 117; blue, 5 }  ,opacity=1 ]  {$T^{2} c_{c}$};
\draw (501,130.4) node [anchor=south] [inner sep=0.75pt]  [color={rgb, 255:red, 208; green, 2; blue, 27 }  ,opacity=1 ]  {$T\alpha _{c}$};
\draw (464.45,132.86) node [anchor=south east] [inner sep=0.75pt]  [color={rgb, 255:red, 189; green, 16; blue, 224 }  ,opacity=1 ]  {$\beta $};

\end{tikzpicture}
\caption{\label{FigConsbeta}Construction of the curve $\beta$ in $S_c$ for the proof of Lemma \ref{Lem:distanceestimates}.}
\end{center}
\end{figure}
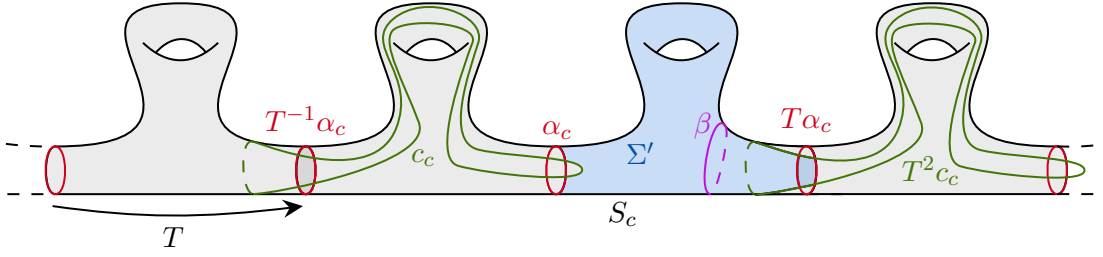

The curve $\beta$ is constructed as follows (see Figure~\ref{FigConsbeta}). Fix a lift $\alpha_c$ of $\alpha$ to $S_c$ and denote by $c_c$ the lift of $c$ which meets $\alpha_c$ but such that no lift of the form $T^{k}(c_c)$, with $k<0$, meets $\alpha$. Hence $T^{n+1}(c_c)$ does not meet $\alpha_c$ either. Now take the connected component $\Sigma$ of $S_c \setminus (c_c \cup \alpha_c)$ which contains $T^{k}(\alpha_c)$ for any $k >0$. Consider now the surface $\Sigma'$ which is the bounded connected component of $\Sigma \setminus (T(\alpha_c) \cup T^{n+1}(c_c))$. This surface has one boundary component which is made of pieces of $\alpha_c$ and $c_c$ and another boundary component which is made of pieces of $T(\alpha_c)$ and $T^{n+1}(c_c)$. The interior of this surface does not meet any lift of $\alpha$ and meets only the lifts $T^k(c_c)$ of $c$, with $1 \leq k \leq n$. Let $\beta_c$ be any closed curve isotopic to  $T(\alpha)$ which belongs to the interior of $\Sigma'$. The projection $\beta$ of $\beta_c$ on $S$ will satisfy the wanted requirements.
\end{proof}

\paragraph{Second case:} the rotation set $\rho_{erg}(f)$ is contained in a hyperplane which is not rational. 

Take a (non-rational) vector $v$ of $H_1(S,\mathbb{R}) $ such that $\rho_{erg}(f) \subset v^{\perp}$. Then, by Claim~\ref{Cla:2definitions}, $\rho_{H_1}(f) \subset v^{\perp}$. Also, fix a euclidean norm $\left\|.\right\|$ on $H_1(S)$.

We will approximate this hyperplane by rational hyperplanes and use transverse rotation sets (as in the first case) to prove that $f$ does not act hyperbolically on the fine curve graph of $S$.

Suppose for a contradiction that $f$ acts hyperbolically on the fine curve graph. By definition, for any element $ c  \in C^{\dagger}(S)$
$$ \lim_{n \rightarrow +\infty} d_{C^{\dagger}}(f^{n}(c),c)= \lambda >0$$
and this constant $\lambda$ does not depend on the chosen $c$.

As the set $\rho_{H_1}(f)$ is bounded, there exists a constant $C_1 >0$ such that, for any $\rho \in \rho_{H_1}(f)$ and any vector $x$ in $H_1(S,\mathbb{R})$,
$$ |\rho \wedge x| \leq C_1 \left\| x \right\|.$$

By Dirichlet's theorem, there exists infinitely many integers $q$ such that there exists an integral vector $v_q \in H_1(S,\mathbb{Z})$ satisfying:
$$ \left\| qv-v_q \right\| \leq \frac{2g}{q^{\frac{1}{2g}}}.$$
For any such integer $q$, we obtain that, for any $\rho \in \rho_{H_1}(f)$
$$ |\rho \wedge v_q| \leq C_1 \frac{2g}{q^{\frac{1}{2g}}}.$$
Fix $q$ large enough so that the right-hand side of this inequality is smaller than $\frac{\lambda}{4}$. Also, taking a smaller vector which is collinear to $v_q$ instead of $v_q$ if necessary, we can suppose that $v_q$ is irreducible. Let $c$ be a simple closed curve of $S$ that represents $v_q$.

Then, taking the same notation as in the first case,
$$ \rho^{t}_{v_q}(f) \subset \left[-\frac{\lambda}{4}, \frac{\lambda}{4}\right].$$ 

This implies that all the limit values of the sequence  $$\left(\frac{1}{n} \mathcal{C}_c\big(f^{n}(c)\big) \right)_n$$
belong to the interval $[0, \frac{\lambda}{2}]$.

Lemma \ref{Lem:distanceestimates} implies then that
$$\lim_{n \rightarrow +\infty} \frac{1}{n}d_{C^{\dagger}}(f^{n}(c),c)\leq \frac{\lambda}{2},$$
a contradiction.

\paragraph{Third case:} no hyperplane contains $\rho_{erg}(f)$ and there exist two orthogonal and complementary rational subspaces $E$ and $F$ of $H_1(S, \mathbb{R})$ such that $\rho_{erg}(f) \subset E \cup F$.

%
%

Recall that we have fixed an isotopy $(f_t)_{t\in [0,1]}$ between $\Id$ and $f$.
Let $x$ be a periodic point of $f$ of period $q$. We call \emph{geodesic associated to $x$}, and denote $\gamma_x$, the closed geodesic that is freely homotopic to the closed curve $(f_t(x))_{t\in [0,q]}$ (if this path is non contractible). We also denote $\displaystyle r(x) = r\left(\frac 1q \sum_{i=0}^{q-1}\delta_{f^i(x)}\right)$ the homological rotation vector associated to the periodic point $x$.

\begin{claim}
There exist four periodic points $x_E,y_E,x_F,y_F$ of respective periods $q_{x_E}, q_{y_E}, q_{x_F},q_{y_F}$ such that:
\begin{itemize}
\item $r(x_E),r(y_E)\in E\setminus\{0\}$ and $r(x_F),r(y_F)\in F\setminus\{0\}$;
\item $\gamma_{x_E}$ and $\gamma_{y_E}$ are simple and have a single intersection point, and $\gamma_{x_F}$ and $\gamma_{y_F}$ are simple and have a single intersection point;
\item $(\gamma_{x_E}\cup\gamma_{y_E}) \cap (\gamma_{x_F}\cup\gamma_{y_F}) = \emptyset$.
\end{itemize}
\end{claim}

\begin{proof}
By Lemma~\ref{LemNotContainedHyp}, the set $\mathcal{L}$ appearing in Proposition~\ref{Prop:decomposition} is empty, and each of the subspaces $V_i$ which appear in the decomposition given by Proposition~\ref{Prop:decomposition} is a symplectic subspace of $H_1(S,\mathbb{R})$. This implies that $E$ and $F$ each contain a finite (and nonempty) unions of the $V_i$'s. Moreover, by point 4. of Proposition~\ref{Prop:decomposition}, the subspaces $E$ and $F$ each contain the rotation vector of a periodic orbit: there exist periodic points $x_E$ and $x_F$ of respective periods $q_E$ and $q_F$ such that $r(x_E)\in E\setminus\{0\}$ and $r(x_F)\in F\setminus\{0\}$. 

Let us first show that the closed paths $(f_t(x_E))_{t\in[0,q_E]}$ and $(f_t(x_F))_{t\in[0,q_F]}$ can be supposed to be homotopic to simple closed geodesics. 

Suppose that this is not the case and that for example $(f_t(x_E))_{t\in[0,q_E]}$ is homotopic to a closed geodesic $(\gamma(s))_{s\in \Sp^1}$ with a transverse self-intersection. Let $s_1,s_2\in\Sp^1$ be two successive intersection times: $\gamma(s_1) = \gamma(s_2)$ and $\gamma|_{[s_1,s_2]}$ is simple. We denote $\tilde S$ the universal cover of $S$, and $\tilde f$ the canonical lift of $f$ to $\tilde S$ (\emph{i.e.} the one commuting with deck transformations). 
Let $\tilde\gamma$ be a lift of $\gamma$ to $\tilde S$, and $T_1$ and $T_2$ be the deck transformations of $\tilde S$ associated to respectively $\tilde \gamma|_{[s_1,s_2]}$ and $\tilde\gamma|_{[s_2,s_1]}$. By \cite[Theorem E]{guiheneuf2022homotopic}, there exists $m>0$ such that a dense subset (containing the extremal points) of the segment $[1/m[T_1]_{H_1}, 1/m[T_2]_{H_1}]$ is realised by rotation vectors of periodic orbits of $f$. Note that the line $\langle r(x_E) \rangle=\langle [T_1]_{H_1}+[T_2]_{H_1} \rangle$ meets $[1/m[T_1]_{H_1}, 1/m[T_2]_{H_1}]$. This implies (because $\rho_{erg}(f)\subset E\cup F$) that $[1/m[T_1]_{H_1}, 1/m[T_2]_{H_1}]\subset E$. If $[T_1]_{H_1}\neq 0$ (in other words, if $T_1$ is not a commutator), then there exists a periodic point $x'_E$ realising the rotation vector $1/m[T_1]_{H_1}\in E\setminus\{0\}$. Otherwise, there exists a periodic point $y_E$ realising the rotation vector $1/m[T_2]_{H_1}\in E\setminus\{0\}$, and we can apply the same reasoning as before to the point $y_E$. Iterating this process, we finally find a periodic point $x'_E$ of period $q'_E$, realising a rotation vector in $E\setminus\{0\}$ and such that the closed path $(f_t(x'_E))_{t\in[0,q'_E]}$ is homotopic to simple closed geodesic (the process terminates because the initial closed geodesic associated to $x_E$ has a finite number of self-intersections). 

Denote by $\gamma_{x_E}$ and $\gamma_{x_F}$ the simple closed geodesics homotopic to respectively $(f_t(x_E))_{t\in[0,q_{x_E}]}$ and $(f_t(x_F))_{t\in[0,q_{x_F}]}$. By \cite[Theorem F]{guiheneuf2022homotopic}, these geodesics $\gamma_{x_E}$ and $\gamma_{x_F}$ do not intersect. Indeed, otherwise, denote $\tilde \gamma_{x_E}$ and $\tilde\gamma_{x_F}$ lifts of $\gamma_{x_E}$ and $\gamma_{x_F}$ to $\tilde S$ that intersect, and $T_E$ and $T_F$ deck transformations of $\tilde S$ associated to these lifts of closed geodesics. Then by \cite[Theorem F]{guiheneuf2022homotopic} there would exist $m>0$ such that $1/(2m)\big([T_E]_{H_1}+[T_F]_{H_1}\big) \in \rho_{erg}(f)$, which contradicts the fact that $\rho_{erg}(f)\subset E\cup F$.
\bigskip

Let us now explain how to get the periodic point $y_E$, the obtaining of $y_F$ being identical.

As $r(x_E)\in E\setminus\{0\}$, and as $E$ is symplectic (Lemma~\ref{LemNotContainedHyp}), there exists $r_0\in \rho_{erg}(f)\cap E$ with $r_0\wedge r(x_E)\neq 0$. By Proposition~\ref{Prop:decomposition}, this rotation vector $r_0$ is accumulated by rotation vectors in $E$ of periodic orbits: there exists a periodic point $y_E$ such that $r(y_E)\wedge r(x_E)\neq 0$ and $r(y_E) \in E$. 

Let us first show that one can suppose that the cardinality of $\gamma_{x_E}\cap\gamma_{y_E}$ is 1. 
if this is not the case, we parametrize the closed geodesics $\gamma_{x_E}$ and $\gamma_{y_E}$ by $\Sp^1$. Then there exists $s_{x,1},s_{x,2},s_{y,1}, s_{y,2}\in\Sp^1$ such that $s_{y,1}$ and $s_{y,2}$ are two successive intersection times of $\gamma_{y_E}$ with $\gamma_{x_E}$ such that these intersections have the same orientation, in other words:
\begin{itemize}
\item $\gamma_{x_E}(s_{x,1}) = \gamma_{y_E}(s_{y,1})$ and $\gamma_{x_E}(s_{x,2}) = \gamma_{y_E}(s_{y,2})$;
\item $\gamma_{y_E}|_{(s_{y,1},s_{y,2})} \cap \gamma_{x_E} = \emptyset$;
\item $\gamma_{y_E}$ crosses $\gamma_{x_E}$ either from left to right at both $s_{y,1}$ and $s_{y,2}$, or from right to left at both $s_{y,1}$ and $s_{y,2}$.
\end{itemize} 
The last property comes from the fact that $r(x_E)\wedge r(y_E)\neq 0$.

By \cite[Theorem F]{guiheneuf2022homotopic}, there is a periodic point $y'_E$ of period $q_{y'_E}$ such that its closed path under the isotopy $(f_t(y'_E))_{t\in [0,q_{y'_E}]}$ is homotopic to the closed curve $\gamma_{y_E}|_{[s_{y,1},s_{y,1}+1]}\gamma_{x_E}|_{[s_{x,1},s_{x,1}+1]}$. This closed path has a transverse self-intersection, so by \cite[Theorem E]{guiheneuf2022homotopic} there is a periodic orbit $y''_E$ of period $q_{y''_E}$ such that its closed path under the isotopy $(f_t(y''_E))_{t\in [0,q_{y''_E}]}$ is homotopic to the closed curve $\gamma_{y_E}|_{[s_{y,1},s_{y,2}]}\gamma_{x_E}|_{[s_{x,2},s_{x,1}+1]}$. By the hypothesis made on the orientation of crossings of $\gamma_{y_E}$ and $\gamma_{x_E}$ at $s_{y,1}$ and $s_{y,2}$, this closed curve intersects $\gamma_{x_E}$ exactly once. 

Finally, if the closed geodesic $\gamma_{y''_E}$ is not simple, let us denote $\overline{s_x}, \overline{s_y}\in\Sp^1$ and $\overline{s_{y,-1}}\neq \overline{s_{y,1}}\in\Sp^1$ such that:
\begin{itemize}
\item $\gamma_{x_E}(\overline{s_x}) = \gamma_{y''_E}(\overline{s_y})$;
\item $\gamma_{y''_E}(\overline{s_{y,-1}}) = \gamma_{y''_E}(\overline{s_{y,1}})$;
\item $\overline{s_y}\in [\overline{s_{y,-1}},\overline{s_{y,1}}]$.
\end{itemize}
To find these times, note that the first condition determines uniquely  $\overline{s_y}\in\Sp^1$, and choose $\overline{s_{y,1}}$ the smallest $s\ge \overline{s_y}$ at which $\gamma_{y''_E}$ intersects itself transversally.

Applying \cite[Theorem E]{guiheneuf2022homotopic} again, we get a periodic point $y'''_E$ of $f$ such that $\gamma_{y'''_E}$ intersects $\gamma_{x_E}$ once and has a number of self-intersections strictly smaller than the one of $\gamma_{y''_E}$. 
Iterating this process if necessary, we get a periodic point $y'''_E$ of $f$ such that $\gamma_{y'''_E}$ intersects $\gamma_{x_E}$ once and is simple.
In particular, this implies that $r(x_E)\wedge r(y'''_E)\neq 0$, and so that $r(y'''_E)\in E\setminus\{0\}$. 

The last point of the claim is obtained identically to the proof we have already made of the fact that $\gamma_{x_E}$ and $\gamma_{x_F}$ do not intersect (this is a direct consequence of \cite[Theorem F]{guiheneuf2022homotopic}).
\end{proof}

This is a classical fact (\emph{e.g.} \cite[1.3.3]{Farb-Margalit}) that all pairs of geodesics intersecting only once are homeomorphic\footnote{Note that on the contrary, there are pairs of disjoint geodesics that are not homeomorphic, for example if the first pair is bounding (\emph{i.e.} the union of the geodesics separates the surface) and the other is not.}. This allows to get a  closed geodesic $\gamma$ separating the geodesics $\gamma_{x_E}$ and $\gamma_{y_E}$ from all the other geodesics of $S$ that do not intersect $\gamma_{x_E}\cup\gamma_{y_E}$. In particular, it separates $\gamma_{x_E}$ and $\gamma_{y_E}$ from $\gamma_{x_F}$ and $\gamma_{y_F}$.

\begin{remark}
By using the proofs of the claim and of Lemma~\ref{LemNotContainedHyp}, it is possible to show that, when $\rho_{erg}(f)$ is not contained in any hyperplane, each subspace $V_i$ appearing in the decomposition given by Proposition \ref{Prop:decomposition} has a symplectic basis consisting of vectors with integral coefficients. By \cite{MR0474304}, every symplectic automorphism of $H_1(S,\mathbb{Z})$ is induced by a diffeomorphism of $S$. This implies that, for any $i$, there exists a subsurface $S_i \subset S$ with boundary such that $H_1(S_i,\mathbb{R})=V_i$. Moreover, the surfaces $S_i$ can be chosen to be pairwise disjoint.
\end{remark}

We will use a covering map involving $\gamma$ so that we can define an analogue of a transverse rotation set associated to $\gamma$, similarly to what we did in case 1.

Let us describe how to obtain such a covering map. An example of such covers is given in Figure~\ref{Fig2Covers}. Denote by $S_E$ and $S_F$ the two connected components of $S \setminus \gamma$, where $S_E$ contains $\gamma_{x_E}$ and $S_F$ contains $\gamma_{x_F}$.

To simplify notations, denote $\alpha=\gamma_{x_E}$ and $\beta=\gamma_{x_F}$; they are respectively contained in $S_E$ and $S_F$ and collinear to nontrivial rotation vectors in $\rho_{erg}(f)$. Cut the surface $S$ along those two simple closed curves to obtain a surface $\mathcal{S}$ with four boundary components $\alpha_{1}$, $\alpha_{2}$ (those two first curves correspond to the closed curve $\alpha$ of $S$), $\beta_1$, $\beta_2$. Take two copies $\mathcal{S}_1$ and $\mathcal{S}_2$ of $\mathcal{S}$. For $i=1,2$ denote by $\alpha_{1,i}$,$\alpha_{2,i}$,$\beta_{1,i}$,$\beta_{2,i}$ for $\mathcal{S}_i$ the boundary components which correspond respectively to $\alpha_1$, $\alpha_2$, $\beta_1$, $\beta_2$. Now, glue $\alpha_{1,1}$ to $\alpha_{2,2}$, $\alpha_{1,2}$ to $\alpha_{2,1}$, $\beta_{1,1}$ to $\beta_{2,2}$, $\beta_{1,2}$ to $\beta_{2,1}$, in such a way that corresponding points of $\alpha$ or $\beta$ are glued together, to obtain a closed surface $\hat{S}$. The map $\hat{S} \rightarrow S$ is a degree $2$ covering map. This map can be also obtained as the covering map corresponding to the kernel of the map
$$\begin{array}{rcl}
H_1(S,\Z) &  \longrightarrow & \Z/2\Z \\
h & \longmapsto & h \wedge (a+b)\mod 2,
\end{array}$$
where $a$ and $b$ respectively represent the homology classes of $\alpha$ and $\beta$.

\begin{figure}
\begin{center}

\tikzset{every picture/.style={line width=0.9pt}} 

\begin{tikzpicture}[x=0.75pt,y=0.75pt,yscale=-1.2,xscale=1.3]

\draw  [fill={rgb, 255:red, 0; green, 0; blue, 0 }  ,fill opacity=0.08 ] (359.04,49.13) .. controls (359.09,44.13) and (378.99,44.76) .. (379.04,48.5) .. controls (379.3,57.39) and (399.09,56.87) .. (399.04,49.21) .. controls (399.2,44.63) and (419.09,44.42) .. (419.04,49.21) .. controls (418.96,83.19) and (393.04,83.39) .. (349.04,83.5) .. controls (344.09,83.45) and (344.51,63.5) .. (349.04,63.5) .. controls (359.58,63.5) and (359.2,57.92) .. (359.04,49.13) -- cycle ;
\draw  [fill={rgb, 255:red, 0; green, 0; blue, 0 }  ,fill opacity=0.08 ] (180.5,49.13) .. controls (180.55,44.13) and (200.45,44.76) .. (200.5,48.5) .. controls (200.76,57.39) and (220.55,56.87) .. (220.5,49.21) .. controls (220.66,44.63) and (240.55,44.42) .. (240.5,49.21) .. controls (240.42,83.19) and (214.5,83.39) .. (170.5,83.5) .. controls (165.55,83.45) and (165.97,63.5) .. (170.5,63.5) .. controls (181.04,63.5) and (180.66,57.92) .. (180.5,49.13) -- cycle ;
\draw  [fill={rgb, 255:red, 0; green, 0; blue, 0 }  ,fill opacity=0.08 ] (532.25,162.54) .. controls (532.15,167.42) and (512.23,167.42) .. (512.25,162.46) .. controls (512.38,152.59) and (492.29,152.21) .. (492.25,162.37) .. controls (492.06,167.25) and (472.23,167.16) .. (472.25,162.29) .. controls (472.29,152.46) and (463.33,142.42) .. (452.33,142.21) .. controls (441.33,141.99) and (432.38,152.44) .. (432.25,162.12) .. controls (432.15,167) and (412.15,166.91) .. (412.25,162.04) .. controls (412.29,152.2) and (392.11,152.09) .. (392.25,161.95) .. controls (392.15,166.91) and (372.15,166.91) .. (372.25,161.87) .. controls (372.29,151.92) and (373.75,121.71) .. (452.42,122.21) .. controls (531.08,122.7) and (532.29,152.5) .. (532.25,162.54) -- cycle ;
\draw  [fill={rgb, 255:red, 0; green, 0; blue, 0 }  ,fill opacity=0.08 ] (372.25,161.95) .. controls (372.33,157.07) and (392.25,156.99) .. (392.25,161.95) .. controls (392.16,171.81) and (412.25,172.12) .. (412.25,161.95) .. controls (412.42,157.07) and (432.25,157.07) .. (432.25,161.95) .. controls (432.25,171.78) and (441.25,181.78) .. (452.25,181.95) .. controls (463.25,182.12) and (472.16,171.63) .. (472.25,161.95) .. controls (472.33,157.07) and (492.33,157.07) .. (492.25,161.95) .. controls (492.25,171.78) and (512.43,171.81) .. (512.25,161.95) .. controls (512.33,156.99) and (532.33,156.91) .. (532.25,161.95) .. controls (532.25,171.9) and (530.92,202.12) .. (452.25,201.95) .. controls (373.58,201.78) and (372.25,172) .. (372.25,161.95) -- cycle ;
\draw [color={rgb, 255:red, 0; green, 0; blue, 0 }  ,draw opacity=1 ][fill={rgb, 255:red, 0; green, 0; blue, 0 }  ,fill opacity=0.08 ]   (354.6,293.5) .. controls (334.4,292.9) and (324,302.9) .. (304.6,303.5) .. controls (256.8,304.5) and (256,234.1) .. (304.6,233.5) .. controls (316.62,233.39) and (335.57,243.53) .. (354.6,243.5) .. controls (384.4,242.5) and (394.4,233.7) .. (404.6,233.5) .. controls (452.45,234.07) and (455.6,304.5) .. (404.6,303.5) .. controls (394.8,303.7) and (376,293.7) .. (354.6,293.5) -- cycle ;
\draw [color={rgb, 255:red, 144; green, 19; blue, 254 }  ,draw opacity=1 ] [dash pattern={on 4.5pt off 4.5pt}]  (354.6,243.5) .. controls (361.89,243.36) and (359.89,293.07) .. (354.6,293.5) ;
\draw [draw opacity=0][fill={rgb, 255:red, 255; green, 255; blue, 255 }  ,fill opacity=1 ]   (299.78,270.21) .. controls (307.52,261.54) and (315.18,264.13) .. (319.69,270.21) .. controls (313.35,274.46) and (306.02,274.79) .. (299.78,270.21) -- cycle ;
\draw    (294.6,266.16) .. controls (304.46,276.09) and (314.89,275.8) .. (324.6,266.16) ;
\draw    (299.78,270.21) .. controls (306.64,262.28) and (314.55,263.28) .. (319.69,270.21) ;

\draw [draw opacity=0][fill={rgb, 255:red, 255; green, 255; blue, 255 }  ,fill opacity=1 ]   (389.78,270.31) .. controls (397.52,261.64) and (405.18,264.23) .. (409.69,270.31) .. controls (403.35,274.56) and (396.02,274.89) .. (389.78,270.31) -- cycle ;
\draw    (384.6,266.26) .. controls (394.46,276.19) and (404.89,275.9) .. (414.6,266.26) ;
\draw    (389.78,270.31) .. controls (396.64,262.38) and (404.55,263.38) .. (409.69,270.31) ;

\draw [color={rgb, 255:red, 144; green, 19; blue, 254 }  ,draw opacity=1 ]   (354.6,243.5) .. controls (346.74,243.64) and (347.6,293.36) .. (354.6,293.5) ;
\draw [shift={(349.08,271.11)}, rotate = 269.68] [fill={rgb, 255:red, 144; green, 19; blue, 254 }  ,fill opacity=1 ][line width=0.08]  [draw opacity=0] (7.14,-3.43) -- (0,0) -- (7.14,3.43) -- (4.74,0) -- cycle    ;
\draw [color={rgb, 255:red, 208; green, 2; blue, 27 }  ,draw opacity=1 ]   (308.46,273.57) .. controls (301.6,273.93) and (299.31,304.21) .. (307.6,303.07) ;
\draw [shift={(302.29,291.33)}, rotate = 273.71] [fill={rgb, 255:red, 208; green, 2; blue, 27 }  ,fill opacity=1 ][line width=0.08]  [draw opacity=0] (7.14,-3.43) -- (0,0) -- (7.14,3.43) -- (4.74,0) -- cycle    ;
\draw [color={rgb, 255:red, 208; green, 2; blue, 27 }  ,draw opacity=1 ] [dash pattern={on 4.5pt off 4.5pt}]  (308.46,273.57) .. controls (313.86,272.97) and (316.2,301.8) .. (307.6,303.07) ;
\draw [color={rgb, 255:red, 65; green, 117; blue, 5 }  ,draw opacity=1 ] [dash pattern={on 4.5pt off 4.5pt}]  (401.17,273.29) .. controls (395.31,274.43) and (397.74,304.71) .. (406.31,303.43) ;
\draw [color={rgb, 255:red, 65; green, 117; blue, 5 }  ,draw opacity=1 ]   (401.17,273.29) .. controls (409.74,271.71) and (413.17,301.57) .. (406.31,303.43) ;
\draw [shift={(409.83,289.78)}, rotate = 261.55] [fill={rgb, 255:red, 65; green, 117; blue, 5 }  ,fill opacity=1 ][line width=0.08]  [draw opacity=0] (7.14,-3.43) -- (0,0) -- (7.14,3.43) -- (4.74,0) -- cycle    ;
\draw  [fill={rgb, 255:red, 0; green, 0; blue, 0 }  ,fill opacity=0.08 ] (162,167.5) .. controls (162.08,162.63) and (182,162.54) .. (182,167.5) .. controls (181.91,177.36) and (202,177.67) .. (202,167.5) .. controls (202.17,162.63) and (222,162.63) .. (222,167.5) .. controls (222,177.33) and (231,187.33) .. (242,187.5) .. controls (253,187.67) and (261.91,177.18) .. (262,167.5) .. controls (262.08,162.63) and (282.08,162.63) .. (282,167.5) .. controls (282,177.33) and (302.18,177.36) .. (302,167.5) .. controls (302.08,162.54) and (322.08,162.46) .. (322,167.5) .. controls (322,177.45) and (320.67,207.67) .. (242,207.5) .. controls (163.33,207.33) and (162,177.55) .. (162,167.5) -- cycle ;
\draw [color={rgb, 255:red, 208; green, 2; blue, 27 }  ,draw opacity=1 ]   (162,167.5) .. controls (162.08,172.71) and (182.17,172.54) .. (182,167.5) ;
\draw [shift={(167.99,171.06)}, rotate = 2.01] [fill={rgb, 255:red, 208; green, 2; blue, 27 }  ,fill opacity=1 ][line width=0.08]  [draw opacity=0] (7.14,-3.43) -- (0,0) -- (7.14,3.43) -- (4.74,0) -- cycle    ;
\draw [color={rgb, 255:red, 65; green, 117; blue, 5 }  ,draw opacity=1 ]   (302,167.5) .. controls (302.08,162.54) and (322.17,162.63) .. (322,167.5) ;
\draw [color={rgb, 255:red, 144; green, 19; blue, 254 }  ,draw opacity=1 ]   (242,187.5) .. controls (235.17,187.58) and (236.17,207.42) .. (242,207.5) ;
\draw [shift={(237.45,200.02)}, rotate = 269.71] [fill={rgb, 255:red, 144; green, 19; blue, 254 }  ,fill opacity=1 ][line width=0.08]  [draw opacity=0] (7.14,-3.43) -- (0,0) -- (7.14,3.43) -- (4.74,0) -- cycle    ;
\draw [color={rgb, 255:red, 144; green, 19; blue, 254 }  ,draw opacity=1 ] [dash pattern={on 4.5pt off 4.5pt}]  (242,187.5) .. controls (247.17,187.25) and (247.67,207.42) .. (242,207.5) ;
\draw [color={rgb, 255:red, 65; green, 117; blue, 5 }  ,draw opacity=1 ]   (262,167.5) .. controls (262.08,162.54) and (282.17,162.63) .. (282,167.5) ;
\draw [color={rgb, 255:red, 208; green, 2; blue, 27 }  ,draw opacity=1 ]   (202,167.5) .. controls (202.08,162.54) and (222.17,162.63) .. (222,167.5) ;
\draw [color={rgb, 255:red, 208; green, 2; blue, 27 }  ,draw opacity=1 ]   (162,167.5) .. controls (162.08,162.54) and (182.17,162.63) .. (182,167.5) ;
\draw [color={rgb, 255:red, 208; green, 2; blue, 27 }  ,draw opacity=1 ]   (202,167.5) .. controls (202.08,172.71) and (222.17,172.54) .. (222,167.5) ;
\draw [shift={(214.62,171.22)}, rotate = 180.66] [fill={rgb, 255:red, 208; green, 2; blue, 27 }  ,fill opacity=1 ][line width=0.08]  [draw opacity=0] (7.14,-3.43) -- (0,0) -- (7.14,3.43) -- (4.74,0) -- cycle    ;
\draw [color={rgb, 255:red, 65; green, 117; blue, 5 }  ,draw opacity=1 ]   (262,167.5) .. controls (262.08,172.71) and (282.17,172.54) .. (282,167.5) ;
\draw [shift={(267.99,171.06)}, rotate = 2.01] [fill={rgb, 255:red, 65; green, 117; blue, 5 }  ,fill opacity=1 ][line width=0.08]  [draw opacity=0] (7.14,-3.43) -- (0,0) -- (7.14,3.43) -- (4.74,0) -- cycle    ;
\draw [color={rgb, 255:red, 65; green, 117; blue, 5 }  ,draw opacity=1 ]   (302,167.5) .. controls (302.08,172.71) and (322.17,172.54) .. (322,167.5) ;
\draw [shift={(314.62,171.22)}, rotate = 180.66] [fill={rgb, 255:red, 65; green, 117; blue, 5 }  ,fill opacity=1 ][line width=0.08]  [draw opacity=0] (7.14,-3.43) -- (0,0) -- (7.14,3.43) -- (4.74,0) -- cycle    ;
\draw  [fill={rgb, 255:red, 0; green, 0; blue, 0 }  ,fill opacity=0.08 ] (322,152.84) .. controls (321.9,157.72) and (301.98,157.72) .. (302,152.76) .. controls (302.13,142.89) and (282.04,142.51) .. (282,152.67) .. controls (281.81,157.55) and (261.98,157.46) .. (262,152.59) .. controls (262.04,142.76) and (253.08,132.72) .. (242.08,132.51) .. controls (231.08,132.29) and (222.13,142.74) .. (222,152.42) .. controls (221.9,157.3) and (201.9,157.21) .. (202,152.34) .. controls (202.04,142.5) and (181.86,142.39) .. (182,152.25) .. controls (181.9,157.21) and (161.9,157.21) .. (162,152.17) .. controls (162.04,142.22) and (163.5,112.01) .. (242.17,112.51) .. controls (320.83,113) and (322.04,142.8) .. (322,152.84) -- cycle ;
\draw [color={rgb, 255:red, 65; green, 117; blue, 5 }  ,draw opacity=1 ] [dash pattern={on 4.5pt off 4.5pt}]  (322,152.84) .. controls (321.94,147.63) and (301.85,147.71) .. (302,152.76) ;
\draw [color={rgb, 255:red, 208; green, 2; blue, 27 }  ,draw opacity=1 ]   (182,152.25) .. controls (181.9,157.21) and (161.81,157.04) .. (162,152.17) ;
\draw [shift={(169.36,155.78)}, rotate = 359.5] [fill={rgb, 255:red, 208; green, 2; blue, 27 }  ,fill opacity=1 ][line width=0.08]  [draw opacity=0] (7.14,-3.43) -- (0,0) -- (7.14,3.43) -- (4.74,0) -- cycle    ;
\draw [color={rgb, 255:red, 144; green, 19; blue, 254 }  ,draw opacity=1 ] [dash pattern={on 4.5pt off 4.5pt}]  (242.08,132.51) .. controls (248.92,132.45) and (248,112.61) .. (242.17,112.51) ;
\draw [color={rgb, 255:red, 144; green, 19; blue, 254 }  ,draw opacity=1 ]   (242.08,132.51) .. controls (236.92,132.73) and (236.5,112.57) .. (242.17,112.51) ;
\draw [shift={(238.37,126.43)}, rotate = 267.66] [fill={rgb, 255:red, 144; green, 19; blue, 254 }  ,fill opacity=1 ][line width=0.08]  [draw opacity=0] (7.14,-3.43) -- (0,0) -- (7.14,3.43) -- (4.74,0) -- cycle    ;
\draw [color={rgb, 255:red, 208; green, 2; blue, 27 }  ,draw opacity=1 ]   (222,152.42) .. controls (221.9,157.38) and (201.81,157.21) .. (202,152.34) ;
\draw [shift={(216,155.8)}, rotate = 178.21] [fill={rgb, 255:red, 208; green, 2; blue, 27 }  ,fill opacity=1 ][line width=0.08]  [draw opacity=0] (7.14,-3.43) -- (0,0) -- (7.14,3.43) -- (4.74,0) -- cycle    ;
\draw [color={rgb, 255:red, 65; green, 117; blue, 5 }  ,draw opacity=1 ]   (282,152.67) .. controls (281.9,157.63) and (261.81,157.46) .. (262,152.59) ;
\draw [shift={(269.36,156.2)}, rotate = 359.5] [fill={rgb, 255:red, 65; green, 117; blue, 5 }  ,fill opacity=1 ][line width=0.08]  [draw opacity=0] (7.14,-3.43) -- (0,0) -- (7.14,3.43) -- (4.74,0) -- cycle    ;
\draw [color={rgb, 255:red, 65; green, 117; blue, 5 }  ,draw opacity=1 ]   (322,152.84) .. controls (321.9,157.8) and (301.81,157.63) .. (302,152.76) ;
\draw [shift={(316,156.22)}, rotate = 178.21] [fill={rgb, 255:red, 65; green, 117; blue, 5 }  ,fill opacity=1 ][line width=0.08]  [draw opacity=0] (7.14,-3.43) -- (0,0) -- (7.14,3.43) -- (4.74,0) -- cycle    ;
\draw [color={rgb, 255:red, 65; green, 117; blue, 5 }  ,draw opacity=1 ] [dash pattern={on 4.5pt off 4.5pt}]  (282,152.67) .. controls (281.94,147.46) and (261.85,147.55) .. (262,152.59) ;
\draw [color={rgb, 255:red, 208; green, 2; blue, 27 }  ,draw opacity=1 ] [dash pattern={on 4.5pt off 4.5pt}]  (222,152.42) .. controls (221.94,147.21) and (201.85,147.3) .. (202,152.34) ;
\draw [color={rgb, 255:red, 208; green, 2; blue, 27 }  ,draw opacity=1 ] [dash pattern={on 4.5pt off 4.5pt}]  (182,152.25) .. controls (181.94,147.05) and (161.86,147.13) .. (162,152.17) ;
\draw [color={rgb, 255:red, 208; green, 2; blue, 27 }  ,draw opacity=1 ]   (372.25,161.95) .. controls (372.33,167.16) and (392.42,166.99) .. (392.25,161.95) ;
\draw [shift={(378.24,165.51)}, rotate = 2.01] [fill={rgb, 255:red, 208; green, 2; blue, 27 }  ,fill opacity=1 ][line width=0.08]  [draw opacity=0] (7.14,-3.43) -- (0,0) -- (7.14,3.43) -- (4.74,0) -- cycle    ;
\draw [color={rgb, 255:red, 65; green, 117; blue, 5 }  ,draw opacity=1 ]   (512.25,161.95) .. controls (512.33,156.99) and (532.42,157.07) .. (532.25,161.95) ;
\draw [color={rgb, 255:red, 144; green, 19; blue, 254 }  ,draw opacity=1 ]   (452.25,181.95) .. controls (445.42,182.03) and (446.42,201.87) .. (452.25,201.95) ;
\draw [shift={(447.7,194.47)}, rotate = 269.71] [fill={rgb, 255:red, 144; green, 19; blue, 254 }  ,fill opacity=1 ][line width=0.08]  [draw opacity=0] (7.14,-3.43) -- (0,0) -- (7.14,3.43) -- (4.74,0) -- cycle    ;
\draw [color={rgb, 255:red, 144; green, 19; blue, 254 }  ,draw opacity=1 ] [dash pattern={on 4.5pt off 4.5pt}]  (452.25,181.95) .. controls (457.42,181.7) and (457.92,201.87) .. (452.25,201.95) ;
\draw [color={rgb, 255:red, 65; green, 117; blue, 5 }  ,draw opacity=1 ]   (472.25,161.95) .. controls (472.33,156.99) and (492.42,157.07) .. (492.25,161.95) ;
\draw [color={rgb, 255:red, 208; green, 2; blue, 27 }  ,draw opacity=1 ]   (412.25,161.95) .. controls (412.33,156.99) and (432.42,157.07) .. (432.25,161.95) ;
\draw [color={rgb, 255:red, 208; green, 2; blue, 27 }  ,draw opacity=1 ]   (372.25,161.95) .. controls (372.33,156.99) and (392.42,157.07) .. (392.25,161.95) ;
\draw [color={rgb, 255:red, 208; green, 2; blue, 27 }  ,draw opacity=1 ]   (412.25,161.95) .. controls (412.33,167.16) and (432.42,166.99) .. (432.25,161.95) ;
\draw [shift={(424.87,165.67)}, rotate = 180.66] [fill={rgb, 255:red, 208; green, 2; blue, 27 }  ,fill opacity=1 ][line width=0.08]  [draw opacity=0] (7.14,-3.43) -- (0,0) -- (7.14,3.43) -- (4.74,0) -- cycle    ;
\draw [color={rgb, 255:red, 65; green, 117; blue, 5 }  ,draw opacity=1 ]   (472.25,161.95) .. controls (472.55,167.58) and (492.47,167.24) .. (492.25,161.95) ;
\draw [shift={(478.21,165.73)}, rotate = 2.26] [fill={rgb, 255:red, 65; green, 117; blue, 5 }  ,fill opacity=1 ][line width=0.08]  [draw opacity=0] (7.14,-3.43) -- (0,0) -- (7.14,3.43) -- (4.74,0) -- cycle    ;
\draw [color={rgb, 255:red, 65; green, 117; blue, 5 }  ,draw opacity=1 ]   (512.25,161.95) .. controls (512.38,168.08) and (532.38,167.24) .. (532.25,161.95) ;
\draw [shift={(524.87,166.06)}, rotate = 180.01] [fill={rgb, 255:red, 65; green, 117; blue, 5 }  ,fill opacity=1 ][line width=0.08]  [draw opacity=0] (7.14,-3.43) -- (0,0) -- (7.14,3.43) -- (4.74,0) -- cycle    ;
\draw [color={rgb, 255:red, 208; green, 2; blue, 27 }  ,draw opacity=1 ] [dash pattern={on 4.5pt off 4.5pt}]  (432.25,162.12) .. controls (432.19,156.91) and (412.1,157) .. (412.25,162.04) ;
\draw  [fill={rgb, 255:red, 0; green, 0; blue, 0 }  ,fill opacity=0.08 ] (340.5,49.8) .. controls (340.4,54.68) and (320.48,54.68) .. (320.5,49.72) .. controls (320.63,39.86) and (300.54,39.47) .. (300.5,49.64) .. controls (300.31,54.51) and (280.48,54.43) .. (280.5,49.55) .. controls (280.54,39.72) and (271.58,29.68) .. (260.58,29.47) .. controls (249.58,29.26) and (240.63,39.7) .. (240.5,49.38) .. controls (240.4,54.26) and (220.4,54.18) .. (220.5,49.3) .. controls (220.54,39.47) and (200.36,39.35) .. (200.5,49.22) .. controls (200.4,54.18) and (180.4,54.17) .. (180.5,49.13) .. controls (180.54,39.18) and (182,8.97) .. (260.67,9.47) .. controls (339.33,9.97) and (340.54,39.76) .. (340.5,49.8) -- cycle ;
\draw [color={rgb, 255:red, 208; green, 2; blue, 27 }  ,draw opacity=1 ]   (180.5,49.21) .. controls (180.58,54.42) and (200.67,54.25) .. (200.5,49.21) ;
\draw [shift={(186.49,52.77)}, rotate = 2.01] [fill={rgb, 255:red, 208; green, 2; blue, 27 }  ,fill opacity=1 ][line width=0.08]  [draw opacity=0] (7.14,-3.43) -- (0,0) -- (7.14,3.43) -- (4.74,0) -- cycle    ;
\draw [color={rgb, 255:red, 65; green, 117; blue, 5 }  ,draw opacity=1 ]   (320.5,49.21) .. controls (320.58,44.25) and (340.67,44.34) .. (340.5,49.21) ;
\draw [color={rgb, 255:red, 65; green, 117; blue, 5 }  ,draw opacity=1 ]   (280.5,49.21) .. controls (280.58,44.25) and (300.67,44.34) .. (300.5,49.21) ;
\draw [color={rgb, 255:red, 208; green, 2; blue, 27 }  ,draw opacity=1 ]   (220.5,49.21) .. controls (220.58,44.25) and (240.67,44.34) .. (240.5,49.21) ;
\draw [color={rgb, 255:red, 208; green, 2; blue, 27 }  ,draw opacity=1 ]   (180.5,49.21) .. controls (180.58,44.25) and (200.67,44.34) .. (200.5,49.21) ;
\draw [color={rgb, 255:red, 208; green, 2; blue, 27 }  ,draw opacity=1 ]   (220.5,49.21) .. controls (220.58,54.42) and (240.67,54.25) .. (240.5,49.21) ;
\draw [shift={(233.12,52.94)}, rotate = 180.66] [fill={rgb, 255:red, 208; green, 2; blue, 27 }  ,fill opacity=1 ][line width=0.08]  [draw opacity=0] (7.14,-3.43) -- (0,0) -- (7.14,3.43) -- (4.74,0) -- cycle    ;
\draw [color={rgb, 255:red, 65; green, 117; blue, 5 }  ,draw opacity=1 ]   (280.5,49.21) .. controls (280.8,54.84) and (300.72,54.5) .. (300.5,49.21) ;
\draw [shift={(286.46,53)}, rotate = 2.26] [fill={rgb, 255:red, 65; green, 117; blue, 5 }  ,fill opacity=1 ][line width=0.08]  [draw opacity=0] (7.14,-3.43) -- (0,0) -- (7.14,3.43) -- (4.74,0) -- cycle    ;
\draw [color={rgb, 255:red, 65; green, 117; blue, 5 }  ,draw opacity=1 ]   (320.5,49.21) .. controls (320.63,55.34) and (340.63,54.5) .. (340.5,49.21) ;
\draw [shift={(333.12,53.32)}, rotate = 180.01] [fill={rgb, 255:red, 65; green, 117; blue, 5 }  ,fill opacity=1 ][line width=0.08]  [draw opacity=0] (7.14,-3.43) -- (0,0) -- (7.14,3.43) -- (4.74,0) -- cycle    ;
\draw [color={rgb, 255:red, 208; green, 2; blue, 27 }  ,draw opacity=1 ] [dash pattern={on 4.5pt off 4.5pt}]  (240.5,49.38) .. controls (240.44,44.18) and (220.35,44.26) .. (220.5,49.3) ;
\draw  [fill={rgb, 255:red, 0; green, 0; blue, 0 }  ,fill opacity=0.08 ] (280.5,49.13) .. controls (280.55,44.13) and (300.45,44.76) .. (300.5,48.5) .. controls (300.76,57.39) and (320.55,56.87) .. (320.5,49.21) .. controls (320.66,44.63) and (340.55,44.42) .. (340.5,49.21) .. controls (340.73,58.58) and (341.04,63.5) .. (350.5,63.5) .. controls (355.65,63.35) and (355.81,83.5) .. (350.5,83.5) .. controls (295.81,83.35) and (280.73,73.35) .. (280.5,49.13) -- cycle ;
\draw [color={rgb, 255:red, 144; green, 19; blue, 254 }  ,draw opacity=1 ]   (171.44,63.5) .. controls (164.6,63.58) and (165.6,83.42) .. (171.44,83.5) ;
\draw [shift={(166.96,69.43)}, rotate = 91.41] [fill={rgb, 255:red, 144; green, 19; blue, 254 }  ,fill opacity=1 ][line width=0.08]  [draw opacity=0] (7.14,-3.43) -- (0,0) -- (7.14,3.43) -- (4.74,0) -- cycle    ;
\draw [color={rgb, 255:red, 144; green, 19; blue, 254 }  ,draw opacity=1 ] [dash pattern={on 4.5pt off 4.5pt}]  (171.44,63.5) .. controls (176.6,63.25) and (177.1,83.42) .. (171.44,83.5) ;
\draw [color={rgb, 255:red, 144; green, 19; blue, 254 }  ,draw opacity=1 ]   (349.97,63.5) .. controls (343.14,63.58) and (344.14,83.42) .. (349.97,83.5) ;
\draw [shift={(345.5,69.43)}, rotate = 91.41] [fill={rgb, 255:red, 144; green, 19; blue, 254 }  ,fill opacity=1 ][line width=0.08]  [draw opacity=0] (7.14,-3.43) -- (0,0) -- (7.14,3.43) -- (4.74,0) -- cycle    ;
\draw [color={rgb, 255:red, 144; green, 19; blue, 254 }  ,draw opacity=1 ] [dash pattern={on 4.5pt off 4.5pt}]  (349.97,63.5) .. controls (355.14,63.25) and (355.64,83.42) .. (349.97,83.5) ;
\draw  [fill={rgb, 255:red, 0; green, 0; blue, 0 }  ,fill opacity=0.08 ] (519.04,49.8) .. controls (518.93,54.68) and (499.02,54.68) .. (499.04,49.72) .. controls (499.17,39.86) and (479.08,39.47) .. (479.04,49.64) .. controls (478.85,54.51) and (459.02,54.43) .. (459.04,49.55) .. controls (459.08,39.72) and (450.12,29.68) .. (439.12,29.47) .. controls (428.12,29.26) and (419.17,39.7) .. (419.04,49.38) .. controls (418.93,54.26) and (398.93,54.18) .. (399.04,49.3) .. controls (399.08,39.47) and (378.9,39.35) .. (379.04,49.22) .. controls (378.93,54.18) and (358.93,54.17) .. (359.04,49.13) .. controls (359.08,39.18) and (360.54,8.97) .. (439.21,9.47) .. controls (517.87,9.97) and (519.08,39.76) .. (519.04,49.8) -- cycle ;
\draw [color={rgb, 255:red, 208; green, 2; blue, 27 }  ,draw opacity=1 ]   (359.04,49.21) .. controls (359.12,54.42) and (379.2,54.25) .. (379.04,49.21) ;
\draw [shift={(365.03,52.77)}, rotate = 2.01] [fill={rgb, 255:red, 208; green, 2; blue, 27 }  ,fill opacity=1 ][line width=0.08]  [draw opacity=0] (7.14,-3.43) -- (0,0) -- (7.14,3.43) -- (4.74,0) -- cycle    ;
\draw [color={rgb, 255:red, 65; green, 117; blue, 5 }  ,draw opacity=1 ]   (499.04,49.21) .. controls (499.12,44.25) and (519.2,44.34) .. (519.04,49.21) ;
\draw [color={rgb, 255:red, 65; green, 117; blue, 5 }  ,draw opacity=1 ]   (459.04,49.21) .. controls (459.12,44.25) and (479.2,44.34) .. (479.04,49.21) ;
\draw [color={rgb, 255:red, 208; green, 2; blue, 27 }  ,draw opacity=1 ]   (399.04,49.21) .. controls (399.12,44.25) and (419.2,44.34) .. (419.04,49.21) ;
\draw [color={rgb, 255:red, 208; green, 2; blue, 27 }  ,draw opacity=1 ]   (359.04,49.21) .. controls (359.12,44.25) and (379.2,44.34) .. (379.04,49.21) ;
\draw [color={rgb, 255:red, 208; green, 2; blue, 27 }  ,draw opacity=1 ]   (399.04,49.21) .. controls (399.12,54.42) and (419.2,54.25) .. (419.04,49.21) ;
\draw [shift={(411.66,52.94)}, rotate = 180.66] [fill={rgb, 255:red, 208; green, 2; blue, 27 }  ,fill opacity=1 ][line width=0.08]  [draw opacity=0] (7.14,-3.43) -- (0,0) -- (7.14,3.43) -- (4.74,0) -- cycle    ;
\draw [color={rgb, 255:red, 65; green, 117; blue, 5 }  ,draw opacity=1 ]   (459.04,49.21) .. controls (459.34,54.84) and (479.25,54.5) .. (479.04,49.21) ;
\draw [shift={(465,53)}, rotate = 2.26] [fill={rgb, 255:red, 65; green, 117; blue, 5 }  ,fill opacity=1 ][line width=0.08]  [draw opacity=0] (7.14,-3.43) -- (0,0) -- (7.14,3.43) -- (4.74,0) -- cycle    ;
\draw [color={rgb, 255:red, 65; green, 117; blue, 5 }  ,draw opacity=1 ]   (499.04,49.21) .. controls (499.17,55.34) and (519.17,54.5) .. (519.04,49.21) ;
\draw [shift={(511.66,53.32)}, rotate = 180.01] [fill={rgb, 255:red, 65; green, 117; blue, 5 }  ,fill opacity=1 ][line width=0.08]  [draw opacity=0] (7.14,-3.43) -- (0,0) -- (7.14,3.43) -- (4.74,0) -- cycle    ;
\draw [color={rgb, 255:red, 208; green, 2; blue, 27 }  ,draw opacity=1 ] [dash pattern={on 4.5pt off 4.5pt}]  (419.04,49.38) .. controls (418.98,44.18) and (398.89,44.26) .. (399.04,49.3) ;
\draw  [fill={rgb, 255:red, 0; green, 0; blue, 0 }  ,fill opacity=0.08 ] (459.04,49.13) .. controls (459.09,44.13) and (478.99,44.76) .. (479.04,48.5) .. controls (479.3,57.39) and (499.09,56.87) .. (499.04,49.21) .. controls (499.2,44.63) and (519.09,44.42) .. (519.04,49.21) .. controls (519.27,58.58) and (519.58,63.5) .. (529.04,63.5) .. controls (534.19,63.35) and (534.35,83.5) .. (529.04,83.5) .. controls (474.35,83.35) and (459.27,73.35) .. (459.04,49.13) -- cycle ;
\draw [color={rgb, 255:red, 144; green, 19; blue, 254 }  ,draw opacity=1 ]   (528.51,63.5) .. controls (521.68,63.58) and (522.68,83.42) .. (528.51,83.5) ;
\draw [shift={(524.04,69.43)}, rotate = 91.41] [fill={rgb, 255:red, 144; green, 19; blue, 254 }  ,fill opacity=1 ][line width=0.08]  [draw opacity=0] (7.14,-3.43) -- (0,0) -- (7.14,3.43) -- (4.74,0) -- cycle    ;
\draw [color={rgb, 255:red, 144; green, 19; blue, 254 }  ,draw opacity=1 ] [dash pattern={on 4.5pt off 4.5pt}]  (528.51,63.5) .. controls (533.68,63.25) and (534.18,83.42) .. (528.51,83.5) ;
\draw [color={rgb, 255:red, 144; green, 19; blue, 254 }  ,draw opacity=1 ]   (440.39,9.46) .. controls (433.56,9.55) and (434.56,29.38) .. (440.39,29.46) ;
\draw [shift={(435.84,21.98)}, rotate = 269.71] [fill={rgb, 255:red, 144; green, 19; blue, 254 }  ,fill opacity=1 ][line width=0.08]  [draw opacity=0] (7.14,-3.43) -- (0,0) -- (7.14,3.43) -- (4.74,0) -- cycle    ;
\draw [color={rgb, 255:red, 144; green, 19; blue, 254 }  ,draw opacity=1 ] [dash pattern={on 4.5pt off 4.5pt}]  (440.39,9.46) .. controls (445.56,9.21) and (446.06,29.38) .. (440.39,29.46) ;
\draw  [dash pattern={on 4.5pt off 4.5pt}]  (171.44,63.5) .. controls (167.5,63.68) and (164.59,64.41) .. (155.32,62.59) ;
\draw  [dash pattern={on 4.5pt off 4.5pt}]  (170.5,83.5) .. controls (164.77,83.5) and (158.51,83.66) .. (154.83,83.21) ;
\draw  [dash pattern={on 4.5pt off 4.5pt}]  (544.77,83.32) .. controls (542.23,83.32) and (538.77,83.14) .. (528.51,83.5) ;
\draw  [dash pattern={on 4.5pt off 4.5pt}]  (546.95,62.23) .. controls (540.59,63.5) and (536.77,63.5) .. (528.51,63.5) ;
\draw [color={rgb, 255:red, 144; green, 19; blue, 254 }  ,draw opacity=1 ]   (260.52,9.55) .. controls (253.69,9.63) and (254.69,29.47) .. (260.52,29.55) ;
\draw [shift={(255.97,22.06)}, rotate = 269.71] [fill={rgb, 255:red, 144; green, 19; blue, 254 }  ,fill opacity=1 ][line width=0.08]  [draw opacity=0] (7.14,-3.43) -- (0,0) -- (7.14,3.43) -- (4.74,0) -- cycle    ;
\draw [color={rgb, 255:red, 144; green, 19; blue, 254 }  ,draw opacity=1 ] [dash pattern={on 4.5pt off 4.5pt}]  (260.52,9.55) .. controls (265.69,9.3) and (266.19,29.47) .. (260.52,29.55) ;
\draw    (352,91.25) -- (352,122) ;
\draw [shift={(352,125)}, rotate = 270] [fill={rgb, 255:red, 0; green, 0; blue, 0 }  ][line width=0.08]  [draw opacity=0] (8.93,-4.29) -- (0,0) -- (8.93,4.29) -- (5.93,0) -- cycle    ;
\draw    (352,189.25) -- (352,220) ;
\draw [shift={(352,223)}, rotate = 270] [fill={rgb, 255:red, 0; green, 0; blue, 0 }  ][line width=0.08]  [draw opacity=0] (8.93,-4.29) -- (0,0) -- (8.93,4.29) -- (5.93,0) -- cycle    ;

\draw (301.05,287.93) node [anchor=east] [inner sep=0.75pt]  [color={rgb, 255:red, 208; green, 2; blue, 27 }  ,opacity=1 ]  {$\alpha $};
\draw (410.6,284.16) node [anchor=west] [inner sep=0.75pt]  [color={rgb, 255:red, 65; green, 117; blue, 5 }  ,opacity=1 ]  {$\beta $};
\draw (346.79,267.55) node [anchor=east] [inner sep=0.75pt]  [color={rgb, 255:red, 144; green, 19; blue, 254 }  ,opacity=1 ]  {$\gamma $};
\draw (373.05,238) node [anchor=south east] [inner sep=0.75pt]    {$S$};
\draw (326.45,210.48) node [anchor=south] [inner sep=0.75pt]    {$S_{1}$};
\draw (166.22,173.62) node [anchor=north west][inner sep=0.75pt]  [font=\small,color={rgb, 255:red, 208; green, 2; blue, 27 }  ,opacity=1 ]  {$\alpha _{1,1}$};
\draw (213.08,175.65) node [anchor=north] [inner sep=0.75pt]  [font=\small,color={rgb, 255:red, 208; green, 2; blue, 27 }  ,opacity=1 ]  {$\alpha _{2,1}$};
\draw (211.67,144.53) node [anchor=south] [inner sep=0.75pt]  [font=\small,color={rgb, 255:red, 208; green, 2; blue, 27 }  ,opacity=1 ]  {$\alpha _{1,2}$};
\draw (167.75,143.28) node [anchor=south west] [inner sep=0.75pt]  [font=\small,color={rgb, 255:red, 208; green, 2; blue, 27 }  ,opacity=1 ]  {$\alpha _{2,2}$};
\draw (327.75,129.73) node [anchor=south] [inner sep=0.75pt]    {$S_{2}$};
\draw (273.33,171.95) node [anchor=north] [inner sep=0.75pt]  [font=\small,color={rgb, 255:red, 65; green, 117; blue, 5 }  ,opacity=1 ]  {$\beta _{1,1}$};
\draw (318.04,172.2) node [anchor=north east] [inner sep=0.75pt]  [font=\small,color={rgb, 255:red, 65; green, 117; blue, 5 }  ,opacity=1 ]  {$\beta _{2,1}$};
\draw (272.96,142.78) node [anchor=south] [inner sep=0.75pt]  [font=\small,color={rgb, 255:red, 65; green, 117; blue, 5 }  ,opacity=1 ]  {$\beta _{2,2}$};
\draw (318.04,144.53) node [anchor=south east] [inner sep=0.75pt]  [font=\small,color={rgb, 255:red, 65; green, 117; blue, 5 }  ,opacity=1 ]  {$\beta _{1,2}$};
\draw (234.86,193.83) node [anchor=east] [inner sep=0.75pt]  [color={rgb, 255:red, 144; green, 19; blue, 254 }  ,opacity=1 ]  {$\hat{\gamma }$};
\draw (382.7,133.93) node [anchor=south] [inner sep=0.75pt]    {$\hat{S}$};
\draw (421.58,155.57) node [anchor=south] [inner sep=0.75pt]  [font=\small,color={rgb, 255:red, 208; green, 2; blue, 27 }  ,opacity=1 ]  {$\alpha _{1}$};
\draw (378,154.32) node [anchor=south west] [inner sep=0.75pt]  [font=\small,color={rgb, 255:red, 208; green, 2; blue, 27 }  ,opacity=1 ]  {$\alpha _{2}$};
\draw (482.88,154.15) node [anchor=south] [inner sep=0.75pt]  [font=\small,color={rgb, 255:red, 65; green, 117; blue, 5 }  ,opacity=1 ]  {$\beta _{2}$};
\draw (527.62,155.23) node [anchor=south east] [inner sep=0.75pt]  [font=\small,color={rgb, 255:red, 65; green, 117; blue, 5 }  ,opacity=1 ]  {$\beta _{1}$};
\draw (445.78,187.95) node [anchor=east] [inner sep=0.75pt]  [color={rgb, 255:red, 144; green, 19; blue, 254 }  ,opacity=1 ]  {$\hat{\gamma }$};
\draw (175.76,70.5) node [anchor=west] [inner sep=0.75pt]  [color={rgb, 255:red, 144; green, 19; blue, 254 }  ,opacity=1 ]  {$T_{1}^{-1} \gamma _{1}$};
\draw (343.5,69.5) node [anchor=east] [inner sep=0.75pt]  [color={rgb, 255:red, 144; green, 19; blue, 254 }  ,opacity=1 ]  {$\gamma _{1}$};
\draw (349,28.6) node [anchor=south] [inner sep=0.75pt]    {$S_{\gamma }$};
\draw (522.17,70.17) node [anchor=east] [inner sep=0.75pt]  [color={rgb, 255:red, 144; green, 19; blue, 254 }  ,opacity=1 ]  {$T_{1} \gamma _{1}$};
\draw (433.59,18.13) node [anchor=east] [inner sep=0.75pt]  [color={rgb, 255:red, 144; green, 19; blue, 254 }  ,opacity=1 ]  {$\gamma _{2}$};
\draw (255.14,20.51) node [anchor=east] [inner sep=0.75pt]  [color={rgb, 255:red, 144; green, 19; blue, 254 }  ,opacity=1 ]  {$T_{1}^{-1} \gamma _{2}$};
\draw (340.5,149.9) node [anchor=north west][inner sep=0.75pt]    {$\simeq $};

\end{tikzpicture}

\caption{\label{Fig2Covers}The covers used in the proof of the third case.}
\end{center}
\end{figure}
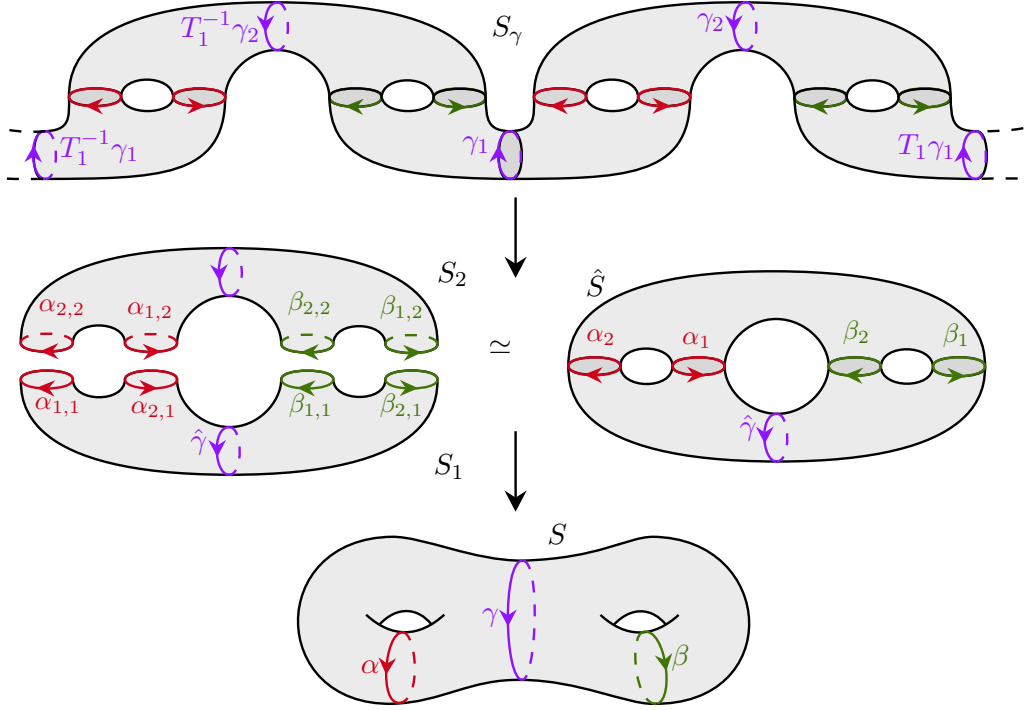

Denote by $\hat{T}$ the nontrivial deck transformation of this covering map. Fix a lift $\hat{\gamma}$ of the curve $\gamma$ to $\hat{S}$ and observe that the closed curve $\hat{\gamma}$ is not separating in $\hat{S}$. Denote by $\hat{f}$ a lift of $f$ to $\hat{S}$. As $\hat{\gamma}$ is nonseparating, we can use the transverse rotation set $\rho^{t}_{\hat{\gamma}}(\hat{f})$ of $\hat{f}$ relative to $\hat{\gamma}$ that we used in the first case.
\bigskip

Suppose first that $\rho^{t}_{\hat{\gamma}}(\hat{f})=\left\{ 0 \right\}$. In this case, we will prove that $f$ does not act hyperbolically on $\mathcal{C}^{\dagger}(S)$. The proof in this case will be similar to the proof in the first case.

Denote by $S_\gamma$ the cyclic covering over $\hat{S}$ corresponding to the map
$$ \begin{array}{rcl}
H_1(\hat{S},\mathbb{R}) & \rightarrow & \mathbb{Z} \\
 a & \mapsto & a \wedge [\hat{\gamma}]
 \end{array} .$$
Denote by $T_1$ a generator of the group of deck transformations of the covering map $S_\gamma \rightarrow \hat{S}$. Fix a lift $\gamma_1$ of $\hat{\gamma}$ to $S_\gamma$ and denote by $\gamma_2$ the lift of $T(\hat{\gamma})$ that lies in the bounded component of $S_\gamma \setminus (\gamma_1 \cup T_1\gamma_1)$.

Finally, for any simple closed curve $\sigma$ which is isotopic to $\gamma$, we define 
$$ C_{\gamma}(\sigma)= \card \big( \left\{n \in \mathbb{Z} \mid  T_1^{n} \gamma_1 \cap \tilde{\sigma} \neq \emptyset \right\}\big) + \card \big( \left\{n \in \mathbb{Z} \mid T_1^{n} \gamma_2 \cap \tilde{\sigma} \neq \emptyset \right\} \big),$$
where $\tilde{\sigma}$ is a lift of $\sigma$ to $S_{\gamma}$. Observe that this quantity does not depend on the chosen lift $\sigma$.

\begin{lemma} \label{Lem:distanceestimates2}
For any simple closed curve $\alpha$ on $S$ which is isotopic to $\gamma$,
$$\mathcal{C}_\gamma(\alpha)+1 \geq d_{C^{\dagger}}(\alpha,\gamma).$$
\end{lemma}

\begin{proof}
The proof of this lemma is almost identical to the proof of Lemma \ref{Lem:distanceestimates}.
\end{proof}

Now, similarly to the first case, as $\rho^{t}_{\hat{\gamma}}(\hat{f})=\left\{ 0 \right\}$,
$$\lim_{n \rightarrow +\infty} \frac{1}{n} \mathcal{C}_\gamma(f^{n}(\gamma))=0.$$
Hence
$$\lim_{n \rightarrow +\infty} \frac{1}{n} d_{C^{\dagger}}(f^{n}(\gamma),\gamma)=0$$
and $f$ does not act hyperbolically on $C^{\dagger}(S)$.
\bigskip

Suppose now that $\rho^{t}_{\hat{\gamma}}(\hat{f}) \neq \left\{ 0 \right\}$. Then we can find an ergodic $\hat{f}$-invariant probability measure with non-zero rotation vector $\hat{c}$. Observe that this rotation vector has a nonzero intersection with the homology class $\hat{a}$ of one of lifts $\hat{\alpha}$ of $\alpha$ to $\hat{S}$ and with the homology class $\hat{b}$ of one of the lifts $\hat{\beta}$ of $\beta$ to $\hat{S}$. By Theorem \ref{Lellouch}, we can suppose, up to changing $\hat{c}$, $\hat{b}$ and $\hat{a}$ for nontrivial colinear vectors, that the homology classes $\hat{c}$, $\frac{1}{2}(\hat{c}+\hat{a})$, $\frac{1}{2}(\hat{c}+\hat{b})$ and $\frac{1}{3}(\hat{c}+\hat{b}+\hat{a})$ are limits of homology rotation vectors realized by periodic orbits. Hence the homology classes $c$ (which is the projection of $\hat{c}$ on $H_1(S,\R)$), $\frac{1}{2}(c+a)$, $\frac{1}{2}(c+b)$ and $\frac{1}{3}(c+b+a)$ are elements of $\overline{\rho_{erg}(f)}$.
Take a homology class $a' \in \rho_{erg}(f) \cap E$ which is realized by a $f$-periodic orbit such that $a' \wedge a \neq 0$ and a homology class $b' \in \rho_{erg}(f) \cap F$ which is realized by a $f$-periodic orbit such that $b' \wedge b \neq 0$.

\begin{lemma}\label{LastLem}
One of the following holds.
\begin{enumerate}
\item $c \wedge a' \neq 0$ and $c \wedge b' \neq 0$.
\item $\frac{1}{2}(c+a) \wedge a' \neq 0$ and $\frac{1}{2}(c+a) \wedge b' \neq 0$.
\item $\frac{1}{2}(c+b) \wedge a' \neq 0$ and $\frac{1}{2}(c+b) \wedge b' \neq 0$.
\item $\frac{1}{3}(c+b+a) \wedge a' \neq 0$ and $\frac{1}{3}(c+b+a) \wedge b' \neq 0$.
\end{enumerate}
\end{lemma}

\begin{proof}
Observe that as $a'\wedge b = a\wedge b' = 0$, the cases of the lemma are equivalent to the following ones:
\begin{enumerate}
\setcounter{enumi}{1}
\item $(c+a) \wedge a' \neq 0$ and $c \wedge b' \neq 0$.
\item $c \wedge a' \neq 0$ and $(c+b) \wedge b' \neq 0$.
\item $(c+a) \wedge a' \neq 0$ and $(c+b) \wedge b' \neq 0$.
\end{enumerate}
We then argue by case disjunction.
\begin{enumerate}[label=\alph*)]
\item If $c \wedge a' \neq 0$ and $c \wedge b' \neq 0$, then 1. holds.
\item If $c \wedge a' \neq 0$ and $c \wedge b' = 0$, then $(c+b) \wedge b' = b \wedge b'\neq 0$, so 3. holds.
\item If $c \wedge a' = 0$ and $c \wedge b' \neq 0$, then $(c+a) \wedge a' = a \wedge a' \neq 0$ so 2. holds.
\item If $c \wedge a' = 0$ and $c \wedge b' = 0$, then $(c+a) \wedge a' = a \wedge a' \neq 0$, and $(c+b) \wedge b' = b \wedge b'\neq 0$, so 4. holds.
\end{enumerate}
\end{proof}

In any case of Lemma~\ref{LastLem}, we have obtained a vector in $\rho_{erg}(f)\setminus\{0\}$ which is not orthogonal to a vector in $E$ and which is not orthogonal to a vector in $F$. This contradicts the hypothesis $\rho_{erg}(f) \subset E \cup F$, with $E$ and $F$ orthogonal. 

This finishes the proof of Theorem~\ref{maintheorem}.
\hfill \qedsymbol

\small

\bibliographystyle{amsalpha}
\bibliography{Biblio}

\end{document}